%% file: ascending_nick_2.tex
\documentclass[11pt]{article}

\usepackage[T1]{fontenc}
\usepackage[utf8]{inputenc}

\usepackage[margin=1in]{geometry}
\usepackage{latexsym}
\usepackage{amssymb,amsmath,amsthm,amsfonts}
\usepackage{mathtools}
\usepackage{mathrsfs}
\usepackage{bm}
\usepackage{mathdots}

\usepackage{graphicx}
\usepackage{epsfig}
\usepackage{psfig}

\usepackage{xcolor}
\usepackage{comment}
\usepackage{algorithmic}

\usepackage{tikz}
\usepackage{tikz-cd}
\usetikzlibrary{calc}
\usepackage[font=small]{caption}
\usepackage[colorlinks=true]{hyperref}

\numberwithin{equation}{section}

\usepackage{newtxtext,newtxmath}

\newtheorem{proposition}{Proposition}[section]

\newtheorem{theorem}{Theorem}[section]

\newtheorem{corollary}{Corollary}[section]

\newcommand{\mcL}{\mathcal{L}}
\newcommand{\mcS}{\mathcal{S}}
\newcommand{\mcZ}{\mathcal{Z}}
\newcommand{\mcC}{\mathcal{C}}
\newcommand{\mcE}{\mathcal{E}}

\title{Ascending Convex Polyominoes}

\author{
Nicholas Beaton\\
\small The University of Melbourne, Australia\\
\small \texttt{nrbeaton@unimelb.edu.au}
\and
Simone Rinaldi\\
\small University of Siena, Italy\\
\small \texttt{rinaldi@unisi.it}
}

\date{}

\begin{document}

\maketitle

\begin{abstract}
Convex polyominoes can be refined according to the number of direction
changes in monotone paths connecting pairs of cells, leading to the
notion of $k$-convexity. In particular, the cases $k=1$ and $k=2$
correspond to $L$-convex and $Z$-convex polyominoes, two well-studied
subclasses of convex polyominoes, with intermediate families such as
centered and $4$-stack polyominoes.
These families exhibit remarkably different combinatorial behaviours, 
suggesting that geometric constraints have a strong impact on the 
nature of the generating function: $L$-convex and centered polyominoes possess rational 
generating functions and growth of order $(2+\sqrt{2})^n$ and $4^n$, respectively, while $Z$-convex, 4-stack, 
and convex polyominoes have algebraic functions and asymptotics 
of order $n4^n$, $\sqrt{n}\,4^n$, and $n4^n$ respectively.

In this paper we investigate the structure of $Z$-convex polyominoes
by introducing a refinement based on the NW- and NE-convexity degrees,
which yields a decomposition into three disjoint subclasses
$\mcC(1,2)$, $\mcC(2,1)$, and $\mcC(2,2)$. To enumerate these families
we introduce ascending polyominoes, admitting a simple geometric
characterization, and construct a generating tree that leads to
functional equations for the corresponding generating functions.
By solving these equations we obtain explicit algebraic generating
functions and the asymptotic growth for all the subclasses.
\end{abstract}

\section{Introduction}

We assume the reader is confident with the concept of a {\em
polyomino}, and with the most important classes of polyominoes,
such as the {\em stack}, the {\em column/row convex}, the {\em directed-convex}, and the {\em convex} polyominoes. For the main
definitions concerning these objects we refer to~\cite{mbm2,convex}.

The notion of $k$-convexity was introduced in~\cite{lconv} as a refinement
of convex polyominoes based on the number of direction changes in
monotone paths connecting pairs of cells. 
A convex polyomino is said to be {\itshape $k$-convex} if any two of its
cells can be joined by an internal monotone path with at most $k$
changes of direction.
The case $k=1$ corresponds to $L$-convex polyominoes
(see Figure~\ref{sample}~(b)) \cite{fpsac,enum,lconv}, while $k=2$ yields the class of
$Z$-convex  polyominoes
(see Figure~\ref{sample}~(d)) \cite{rin}.

A convex polyomino is said to be {\em horizontally centered}  (briefly, {\em centered}), if it contains at least one row touching both the
left and right sides  of its minimal bounding rectangle (see Figure~\ref{sample}~(c)). Finally, a polyomino is called a {\em 4-stack} if it can be decomposed into a central rectangle, referred to as the {\em supporting rectangle}, with four stack polyominoes placed on each side of the rectangle, as graphically shown in Figure~\ref{sample}~(e). 
\begin{figure}[htb]
\begin{center}
\includegraphics[width=150mm]{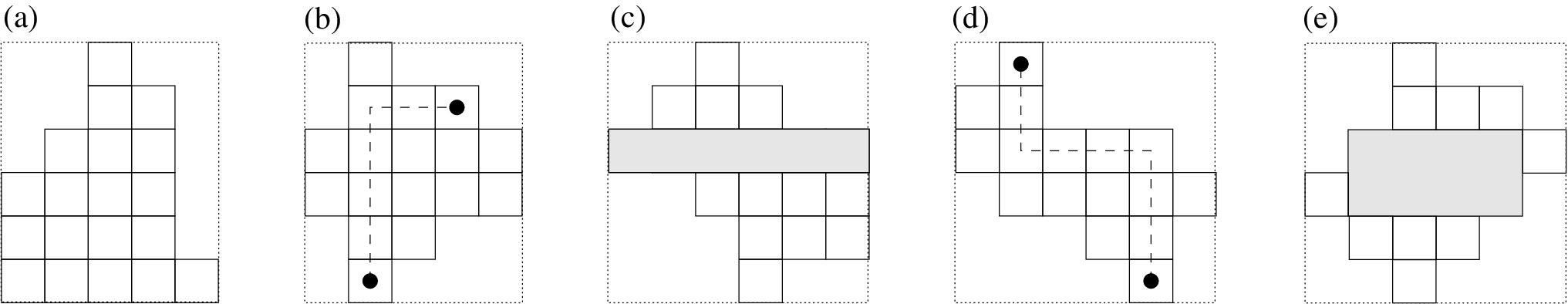} \caption{(a) stack polyomino; (b) $L$-convex polyomino; (c)  centered polyomino; (d) $Z$-convex polyomino; (e) 4-stack polyomino.}\label{sample}
\end{center}
\end{figure}

Let $\mcL_n$ (resp., $\mcE_n$, $\mcZ_n$, $\mcS_n$) be the class of $L$-convex (resp., centered convex, $Z$-convex, 4-stack) polyominoes of semi-perimeter (size) $n$.
Clearly, $\mcL_n \subset \mcE_n \subset \mcS_n \subset \mcZ_n$.  Figure~\ref{sample}~(c) shows a  centered convex polyomino which is not $L$-convex, 
Figure~\ref{sample}~(e) shows a 4-stack polyomino which is not  centered, while Figure~\ref{sample}~(d) depicts a $Z$-convex polyomino which is not 4-stack.

Here are the generating functions of the aforementioned classes.
\begin{enumerate}
 \item The generating function of $L$-convex polyominoes according to the size is \cite{fpsac,enum}:
\begin{equation}
L(t)=\frac{t^{2} \left(t^{2}-2 t+1\right)}{2 t^{2}-4 t+1} \, ,
\end{equation}
and hence the number of $L$-convex polyominoes grows asymptotically as $\ell(n) \sim \left ( \frac{\sqrt{2}-1}{8} \right ) \left( 2+\sqrt{2} \right) ^n.$
 \item The generating function of centered polyominoes according to the size is \cite{rin}:
\begin{equation}
E(t)=\frac{t^2(1-t)(1-3t)}{(1-2t)(1-4t)} \, ,
\end{equation}
and hence the asymptotic growth of centered polyominoes is $e(n) \sim \frac{3}{128}4 ^n.$
\item The generating function of $Z$-convex polyominoes with
respect to the size is  \cite{rin}:
\begin{equation}
Z(t)=\frac{2t^4(1-2t)^2d(t)}
{(1-4t)^2(1-3t)(1-t)}+\frac{t^2(1-6t+10t^2-2t^3-t^4)}
{(1-4t)(1-3t)(1-t)},
\end{equation}
where $
d(t)=\frac12 \left( 1-2t-\sqrt{1-4t} \right)
$
is the generating function of Catalan numbers. Hence, the number of $Z$-convex polyominoes having size $n$ grows asymptotically as $z(n) \sim \frac{n}{384} 4^n$.
\item The generating function of 4-stack polyominoes according to the size is algebraic \cite{4stack}, and precisely it is equal to:
\begin{equation} 
S(t)=\frac{t^2(1-3t)^2}{(1-4t)^{\frac{3}{2}}(1-2t)} \,.
\end{equation}
Then, 4-stack polyominoes have an asymptotic growth of $\sqrt{n}4^n$, which is just between the growth of centered polyominoes and $Z$-convex polyominoes.

\item The generating function of convex polyominoes with
respect to the size is \cite{DV}:
\[
C(t)=\frac{t^2(1-6t+11t^{2}-4t^{3})}{(1-4t)^{2}}
-\frac{4t^{4}}{(1-4t)^{3/2}},
\]
hence the number of convex polyominoes of semiperimeter $n$ grows asymptotically like
$
\frac{n}{128}\,4^n.
$

\end{enumerate}

\medskip

We observe that $Z$-convex and convex polyominoes polyominoes have algebraic  generating functions and the same rough asymptotic behavior, $n4^n$. 
On the other hand, $L$-convex and centered polyominoes have a rational generating function and grow like $(2+\sqrt{2})^n$ and $4^n$, respectively. From these observations, we wish to emphasize the distinguished role of $Z$-convex polyominoes within the hierarchy of convex polyominoes classified according to the convexity degree.

Furthermore, it appears rather surprising that there is no method available in the literature to derive the (algebraic) generating function of $Z$-convex polyominoes using the standard approaches to the enumeration of convex polyominoes, such as Temperley-like methods, representations in terms of algebraic languages~\cite{DV}, decomposition strategies~\cite{mbm2}, or recursive enumeration techniques (for instance, generating trees~\cite{gfgt,convex}, successfully applied to the enumeration of $L$-convex polyominoes~\cite{enum}).

\begin{table}
\begin{center}
\begin{tabular}{l|c|c|c}
  {\bf Class of polyominoes} &{\bf Type of g.f.} & \textbf{OEIS entry} & {\bf Asymptotic growth} \\
  \hline
   & & \\
  {\bf $L$-convex} & rational & \href{https://oeis.org/A003480}{A003480} &  $\left ( \frac{\sqrt{2}-1}{8} \right ) \left( 2+\sqrt{2} \right) ^n$ \\
  & & \\
  {\bf  centered} & rational & \href{https://oeis.org/A049775}{A049775} & $\frac{3}{128} \; 4^n$ \\
  & & \\
  {\bf 4-stack} & algebraic & \href{https://oeis.org/A393099}{A393099} &  $\frac{1}{64 \sqrt{\pi}} \; \sqrt{n} \; 4^n$ \\
  & & \\
  {\bf $Z$-convex} & algebraic & \href{https://oeis.org/A128611}{A128611} & $\frac{n}{384}\; 4^n$ \\
  & & \\
  {\bf convex} & algebraic & \href{https://oeis.org/A005436}{A005436} &  $\frac{n}{128}\; 4^n$ \\
\end{tabular} 
\end{center}
\caption{Asymptotic behaviours of the classes between $L$-convex and $Z$-convex polyominoes. The quantities in the last column are for the number of polyominoes of size $n$.}
\label{tab1}
\end{table}

Other families of polyominoes subject to $k$-convexity constraints have been
investigated in the literature. In particular, for $k \geq 1$, exact
expressions for the generating function with respect to the semi-perimeter
have been obtained for $k$-parallelogram polyominoes and directed
$k$-convex polyominoes~\cite{bussi}. The asymptotic behavior of $k$-convex polyominoes was studied in \cite{micheli}.
Enumeration with respect to the area has been carried out for
$L$-convex polyominoes in~\cite{fpsac}, while for $k \geq 2$ the
asymptotic behaviour of $k$-convex polyominoes has been analyzed
in~\cite{guttmann,guttmann2}.

This paper is organized as follows: 

\begin{description}
\item[-] We present a decomposition of $Z$-convex polyominoes into three
mutually disjoint subclasses, $\mcC(1,2)$, $\mcC(2,1)$, and $\mcC(2,2)$,
based on the notion of NW- and NE-convexity degree,
which extends and refines the classical concept of $k$-convexity~\cite{massazza}.
We focus in particular on the class $\mcC(2,2)$, which turns out to be
the most intriguing of the three, as it is the only family of convex
polyominoes for which both the NW- and NE-convexity degrees exceed~one.
We prove that every polyomino in $\mcC(2,2)$ is in fact a 4-stack polyomino.

\item[-] In order to enumerate the three subclasses, we introduce the
families of ascending and descending (convex) polyominoes,
defined as convex polyominoes that do not belong to $\mcC(2,2)$
and admit an ascending (respectively descending) orientation. In practice:
\[ \mbox{Convex} = \mbox{Ascending} + \mbox{Descending} + \mcC(2,2) ,\]
where ascending and descending polyominoes are equinumerous and their intersection gives $L$-convex polyominoes. 
The family of ascending polyominoes admits a simple
geometric characterization.

\item[-] We construct a generating tree~\cite{gfgt,eco,convex}
for the recursive generation of ascending polyominoes,
which leads to a functional equation satisfied by the corresponding
generating function. By means of standard techniques~\cite{gfgt,convex},
we solve the resulting system and determine the algebraic generating
functions of ascending polyominoes, and consequently of the classes
$\mcC(1,2)$, $\mcC(2,1)$, and $\mcC(2,2)$.
We also derive the asymptotic behaviour of all these families.
\end{description}
Enumeration of ascending polyominoes leads to a deeper comprehension of the structure of convex and $Z$-convex polyominoes and to refine their enumeration. In particular, we show that 
ascending polyominoes, as well as the families $\mcC(1,2)$, $\mcC(2,1)$, have an algebraic generating function and asymptotic growth as $n4^n$. On the other side, the family $\mcC(2,2)$ has an algebraic generating function and  asymptotic growth as $\sqrt{n}4^n$ (the same as $4$-stacks).

\section{Basic definitions and properties}\label{sec:definitions}

\paragraph{Convexity degree and $k$-convex polyominoes} 
A path is monotone if it is made of steps in two directions only: North-steps and East-steps (NE-path), North-steps and West-steps (NW-path), South-steps and East-steps (SE-path) or South-steps and West-steps (SW-path). The authors of \cite{lconv} proved that a polyomino is convex if and only if any two of its cells are joined by a monotone path internal to the polyomino. 
Given $k\geq 0$, a convex polyomino is said to be \emph{k-convex} if any two of its cells can be joined by a monotone path, internal to the polyomino, with at most $k$ changes of direction. Obviously, a $k$-convex polyomino is also $h$-convex,  for every $h \geq k$. We define the {\em degree of convexity} of a convex polyomino $P$ as the smallest $k$ such that $P$ is $k$-convex. 

In~\cite{brocchi}, the notion of $k$-convexity was refined as follows. 
Let $P$ be a convex polyomino and let $a,b$ be two cells. 
We write $a \nwarrow b$ (resp., $a \nearrow b$) if there exists a NW-path (resp., NE-path) joining $a$ to $b$. 
Whenever $a \nearrow b$, the \emph{NE-distance} from $a$ to $b$, denoted by $D_\mathrm{NE}(a,b)$, is defined as the minimum number of direction changes among all NE-paths connecting $a$ to $b$.
The \emph{NE-degree} of convexity of $P$ is then:
\[
D_\mathrm{NE}(P)=\max \{ D_\mathrm{NE}(a,b) : a \nearrow b \}.
\]
Equivalently, it is the smallest integer $k$ such that $D_\mathrm{NE}(a,b)\le k$ for all pairs $a \nearrow b$, and equality holds for at least one such pair.
The quantities $D_\mathrm{NW}(a,b)$ and $D_\mathrm{NW}(P)$ are defined analogously. 
Finally, the (global) degree of convexity of $P$ is
$
D(P)=\max\{D_\mathrm{NE}(P),\,D_\mathrm{NW}(P)\}.
$
The authors of~\cite{brocchi} proved that:

\begin{proposition}\label{prop:dnw}
Let $P$ be a convex polyomino such that $D_\mathrm{NW}(P) < D_\mathrm{NE}(P)$. If $ D_\mathrm{NE}(P) > 2$, then $D_\mathrm{NW}(P) =1$.
\end{proposition}

This property describes how the degrees of convexity in the two directions may behave in a generic convex polyomino. In fact, a convex polyomino having degree of convexity greater than $2$ must be 1-convex in one of the two directions. Hence, a convex polyomino may be either $2$-convex in both directions or $1$-convex in one direction and $k$-convex in the other one ($k$ arbitrary). Figure \ref{fig:small}~(a)~(b) shows a $3$-convex polyomino. 

For $x,y\geq 1$, let us denote by $\mcC(x,y)$ the family of convex polyominoes for which the  NE-degree of convexity is $x$ and the NW-degree of convexity is 
$y$. According to Proposition \ref{prop:dnw}, we have that:
\begin{enumerate}
\item $L$-convex polyominoes are the union of $\mcC(1,1)$, $\mcC(1,0)$, $\mcC(0,1)$, and $\mcC(0,0)$;
\item $Z$-convex polyominoes (not $L$-convex) are given by the disjoint union of $\mcC(1,2)$, $\mcC(2,1)$, and $\mcC(2,2)$;
\item With $k >2$, $k$-convex polyominoes (not $(k-1)$-convex) are given by the disjoint union of  $\mcC(1,k) \cup \mcC(k,1)$;
\item Convex polyominoes $\cal C$ are obtained as the union:
$$ {\cal C}= \bigcup_{k\geq 0} \mcC(k,1) \cup \bigcup_{k\geq 0} \mcC(1,k)  \cup \mcC(0,0) \cup \mcC(2,2) \, .$$
\end{enumerate}

The family $\mcC(2,2)$ turns out to be particularly intriguing within the class of $Z$-convex polyominoes, and more generally within the class of convex polyominoes. For this reason, we aim to study and enumerate this subclass.

\begin{proposition}\label{prop:inclusion}
Every polyomino in $\mcC(2,2)$ is a $4$-stack polyomino. 
\end{proposition}

\begin{proof} 
Let $P \in \mcC(2,2)$ and assume that the south--east corner of the
minimal bounding rectangle of $P$ is placed at $(0,0)$.
Given a row $r$ of $P$, we denote by $L(r)$ (resp.\ $R(r)$) the abscissa
of the left (resp.\ right) endpoint of $r$. Similarly, given a column
$c$, we denote by $U(c)$ (resp.\ $D(c)$) the ordinate of the upper
(resp.\ lower) endpoint of $c$.
Let $r,r'$ be two rows with $r$ below $r'$, and let $c,c'$ be two columns
with $c$ to the left of $c'$. We write $r \nearrow r'$ (resp.\ $r' \nwarrow r$)
if $L(r) < L(r')$ and $R(r) < R(r')$ (resp.\ $L(r') < L(r)$ and $R(r') < R(r)$),
and we write $c \nearrow c'$ if $D(c) < D(c')$ and $U(c) < U(c')$.

Since $P \in \mcC(2,2)$, there exist two pairs of cells $(\alpha,\beta)$
and $(\gamma,\delta)$ such that the NE-distance between $\alpha$ and $\beta$
is equal to $2$, and the NW-distance between $\delta$ and $\gamma$ is equal
to $2$.
By convexity of $P$, there are two possible cases:
\begin{itemize}
\item[(1)] $x_\alpha < x_\delta < x_\gamma < x_\beta$ and
$y_\gamma < y_\alpha < y_\beta < y_\delta$;
\item[(2)] $x_\delta < x_\alpha < x_\beta < x_\gamma$ and
$y_\alpha < y_\gamma < y_\delta < y_\beta$,
\end{itemize}
where, for a cell $\eta$, $(x_\eta,y_\eta)$ are the coordinates of its center.
We assume we are in case (1), since the second one is treated analogously (see Figure~\ref{fig:c221}).

\begin{figure}[htb]
\begin{center}
\scalebox{0.5}{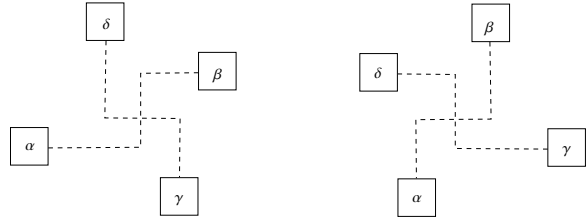}
\caption{The two possible configurations for the points $\alpha,\beta,\gamma,\delta$.}
\label{fig:c221}
\end{center}
\end{figure}

Of all possible choices for $(\alpha,\beta)$ and $(\gamma,\delta)$, we will pick the ones such that the horizontal distance $x_\beta-x_\alpha$ (resp.\ the vertical distance $y_\delta-y_\gamma$) is minimized, and that $\alpha$ (resp.\ $\beta, \gamma, \delta$) is at the bottom (resp.\ top, left, right) end of its column $c_\alpha$ (resp.\ column $c_\beta$, row $r_\gamma$, row $r_\delta$). See Figure~\ref{fig:c22} for a schematic. Note that $c_\alpha \nearrow c_\beta$ and $r_\delta \nwarrow r_\gamma$. 

% Then, there exists a pair of columns $c_\alpha,c_\beta$ such that
% $c_\alpha$ lies to the left of $c_\beta$ and $c_\alpha \nearrow c_\beta$,
% and a pair of rows $r_\gamma,r_\delta$ such that $r_\gamma$ lies below
% $r_\delta$ and $r_\delta \nwarrow r_\gamma$. We assume that
% $c_\alpha,c_\beta$ (resp.\ $r_\gamma,r_\delta$) are chosen so that the
% horizontal (resp.\ vertical) distance between them is minimal.

% Let $r_\alpha$ (resp.\ $r_\beta$) be the row containing $\alpha$
% (resp.\ $\beta$) with minimum (resp.\ maximum) ordinate in $c_\alpha$
% (resp.\ $c_\beta$).
The geometry of the configuration and the convexity of $P$ imply that
\[
L(r_\alpha) < L(r_\beta) \le L(r_\delta) < L(r_\gamma)
\quad \text{and} \quad
R(r_\delta) < R(r_\gamma) \le R(r_\alpha) < R(r_\beta),
\]
where $r_\alpha$ and $r_\beta$ are the rows containing $\alpha$ and $\beta$ respectively.
In particular, we obtain $r_\alpha \nearrow r_\beta$.

Therefore, there exist a row $r_\epsilon$ (between $r_\beta$ and $r_\delta$,
possibly equal to $r_\beta$) and a row $r_\zeta$ (between $r_\gamma$ and
$r_\alpha$, possibly equal to $r_\alpha$) such that $r_\zeta \nearrow r_\epsilon$
and their vertical distance is maximal, in the sense that if $s'$ is any row
strictly above $r_\epsilon$ and $s$ any row strictly below $r_\zeta$, then
$s \nearrow s'$ does not hold.

It follows that the rows above $r_\epsilon$ form an up-stack, while the
rows below $r_\zeta$ form a down-stack.

Consider now the rectangle $\cal R$ whose north--west corner is the left
endpoint of the row $r_\epsilon$ and whose south--east corner is the right
endpoint of the row $r_\zeta$. Since every row $s$ below $r_\zeta$ satisfies
$L(s) \ge L(r_\zeta)$, the columns to the left of $\cal R$ form a left stack.
Similarly, the columns to the right of $\cal R$ form a right stack.

Therefore, $P$ is a $4$-stack polyomino with supporting rectangle $\cal R$.
\end{proof}
\begin{figure}[htb]
\begin{center}
\scalebox{0.45}{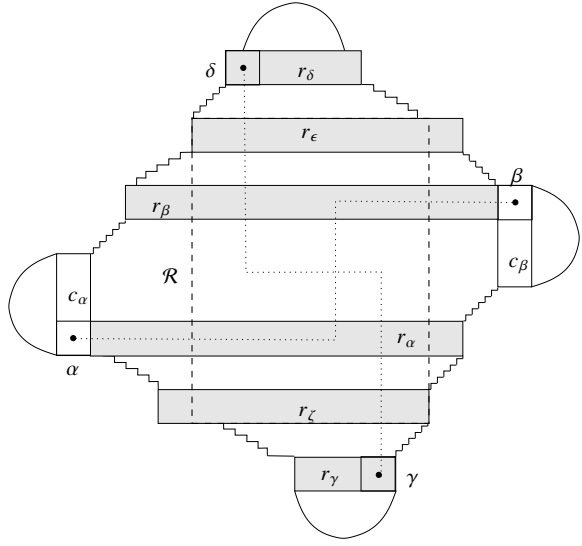}
\caption{Graphical representation of the proof of Proposition~\ref{prop:inclusion}.}
\label{fig:c22}
\end{center}
\end{figure}

A direct enumeration of $\mcC(2,2)$ appears to be non–trivial, since each polyomino in this family is characterized by the simultaneous satisfaction of two existence conditions (namely, the presence of two pairs of rows shifted in the north–east and north–west directions, respectively), which are difficult to control combinatorially. 
It is then more natural to approach the problem indirectly, by studying and enumerating the class of convex polyominoes that do not belong to $\mcC(2,2)$.

\paragraph{Ascending convex polyominoes} 

Let us define: 
\[
{\cal A}=\bigcup_{k\geq 1} \mcC(k,1) \cup \mcC(0,0) \quad \text{and} \quad {\cal D}=\bigcup_{k\geq 1} \mcC(1,k) \cup \mcC(0,0) 
\]
as the family of \emph{ascending} (resp.\ \emph{descending}) convex polyominoes. For any given size, the number of ascending and descending convex polyominoes is equal, and their intersection corresponds to the family of \(L\)-convex polyominoes. Therefore, we obtain the following combinatorial identity:

\begin{proposition}
For every \(n\geq 2\), the number \(c(n)\) of convex polyominoes of size $n$ is given by:
\[
c(n) = 2a(n) + k(n) - \ell(n) \, ,
\]
where \(a(n)\) (resp., \(k(n)\), \(\ell(n)\)) denotes the number of ascending convex (resp.~\(\mcC(2,2)\), \(L\)-convex) polyominoes of size \(n\).
\end{proposition}

Figure~\ref{fig:small} (c) shows a centered ascending polyomino  $\mcC(2,1)$; (d) shows a $4$-stack polyomino in $\mcC(2,2)$; (e) shows (non  centered) polyominoes in $\mcC(2,1)$. 

In this paper, we focus on the enumeration of ascending convex polyominoes (briefly, ascending polyominoes), which, as a neat consequence, leads to the enumeration of the class \(\mcC(2,2)\) (and hence of \(\mcC(1,2)\) and \(\mcC(2,1)\)). This will contribute to a deeper understanding of both \(Z\)-convex polyominoes and convex polyominoes, and provide a cleaner enumeration result than that presented in \cite{rin}.\footnote{The name ``ascending polyominoes'' is also used in a related context \cite{guttmann} for a class of $Z$-convex polyominoes, but the definition adopted there does not coincide with the one studied in the present work.}

\begin{comment}
\begin{figure}[htb]
\begin{center}
\includegraphics[width=110mm]{skew.eps}\caption{(a) A polyomino in $C(2,2)$. There are (at least) two cells that can be uniquely connected by means of a path $E^hN^kE^{h'}$ and (at least) two cells  that can be uniquely connected by means of a path $N^kW^{h}N^{k'}$, $h,h',k,k'\geq 1$. (b) An ascending  polyomino, in $C(2,1)$, where the rows $r,s,t$ are such that $(r,s)\in I$ and $(t,s) \in I$, and $(r,t)\in N$.}\label{fig:skew}
\end{center}
\end{figure}
\end{comment}

\begin{figure}[htb]
\begin{center}
\includegraphics[width=140mm]{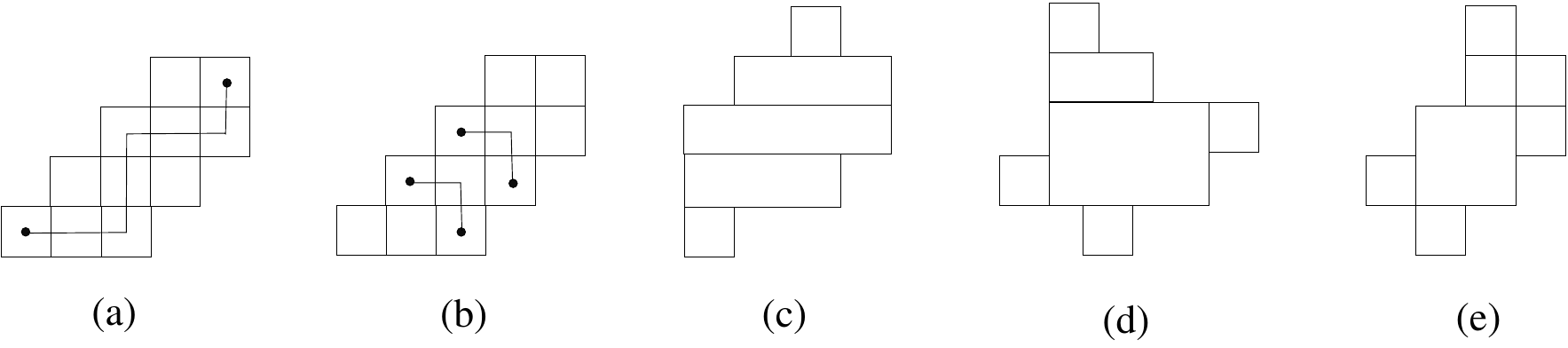}\caption{(a),(b) Convex polyomino in $\mcC(3,1)$; (c) Centered ascending polyomino; (d) Non-centered polyomino in $\mcC(2,2)$. (e) Non-centered ascending polyomino in $\mcC(2,1)$.}\label{fig:small}
\end{center}
\end{figure}

\begin{comment}
    Here is a graphical representation of the decompositions we dealt with in this section.

\begin{figure}
    \centering
    \scalebox{0.25}{\input{schema3.pdftex_t}} 
    \caption{The decomposition of convex polyominoes in the famlies $C(x,y)$, and the family of ascending polyominoes.}
    \label{fig:enter-label}
\end{figure}
\end{comment}

Ascending polyominoes can be characterized geometrically in a straightforward manner, as outlined below, by examining the relative positioning of the rows. This characterization reveals that they naturally extend $L$-convex polyominoes.

%For a generic row  $r$, we denote $x_r,x'_r$ the abscissas of the left and right side of $r$. 

We define the following relations on the rows of a convex polyomino:

\begin{description}
\item{i)} the relation $\preceq$, ({\em row inclusion}). Given two rows $s,t$, we say that $s \preceq t$ if:
$L(t) \leq L(s)$ and $R(s) \leq R(t)$,
\item{ii)} the relation $\nearrow$, ({\em north-east shift}). Given two rows $s,t$, such that $s$ is below $t$, we say that $s \nearrow t$  if: $L(s) < L(t)$ and $R(s) < R(t)$. 
\end{description}
\begin{comment}
\begin{figure}[htb]
\begin{center}
\includegraphics[width=60mm]{inclusion}\caption{Two rows $s,t$: (a) $s \preceq t$. (b) $s \nearrow t$.}\label{inclusion}
\end{center}
\end{figure}

\end{comment}

\begin{proposition}\label{def:neskew}
An ascending polyomino is a convex polyomino such that for every pair of rows $r_1,r_2$, with $r_1$ lying below $r_2$, we have at least one of $r_1 \preceq r_2$, $r_2 \preceq r_1$, or $r_1 \nearrow r_2$.
\end{proposition}

At the end of this section, we would like to emphasize that the family of ascending polyominoes properly includes some well-known families of polyominoes in the literature, characterized by directional properties, such as the family of parallelogram polyominoes and the family of directed-convex polyominoes \cite{mbm2,convex}.

\section{Generation of ascending polyominoes}

We partition the family of ascending polyominoes into two subclasses:
\emph{horizontally centered} (or simply \emph{centered}) ascending polyominoes, denoted by ${\cal H}$,
and \emph{non-centered} ascending polyominoes, denoted by ${\cal N}$.  

Moreover, an ascending polyomino $P$ is said to be \emph{rectangular} if the topmost cell
of its rightmost column attains the maximum height among all cells of $P$ (see Figure~\ref{label}~(b)).

\subsection{Labeling ascending polyominoes}

Let $P\in{\cal H}$ be a centered ascending polyomino.
The set of rows that extend from the left side to the right side of the minimal bounding rectangle of $P$ forms a supporting rectangle of minimal height, which we call the \emph{base} of $P$.

We denote by $b(P)$ (resp.\ $\ell(P)$) the height (resp.\ width) of the base of $P$.
We say that $P$ is a \emph{flipped stack} if it contains no cell strictly above its base.

To each centered ascending polyomino $P \in \mathcal{H}$  we associate a label
$(b,w,r)$ if $P$ is not rectangular, or $(b,w,r)'$ if $P$ is rectangular, with $b\ge1$ and $w,r\ge0$,
which encodes the combinatorial data governing the growth of $P$
(see Figure~\ref{label}):
\begin{description}
\item[$b$:] the height of the base;
\item[$w$:] if $P$ is a flipped stack, then $w=\ell(P)$; otherwise, let $B_u$ denote
the row immediately above the base, and define $w$ as the horizontal distance between $B_u$
and the left side of the minimal bounding rectangle of~$P$;
\item[$r$:] the number of cells in the last column strictly above the base.
\end{description}

Observe that $w>0$ if and only if no cell lies above the upper-left corner of the base
(see Figure~\ref{label}~(a), (b), (c)). 
Otherwise, $w=0$ and $r=0$ (see Figure~\ref{label}~(d)). Moreover, we observe that a flipped stack is always rectangular, with $r=0$ (see Figure~\ref{label}~(c)).

\begin{figure}[htb]
\begin{center}
\scalebox{0.33}{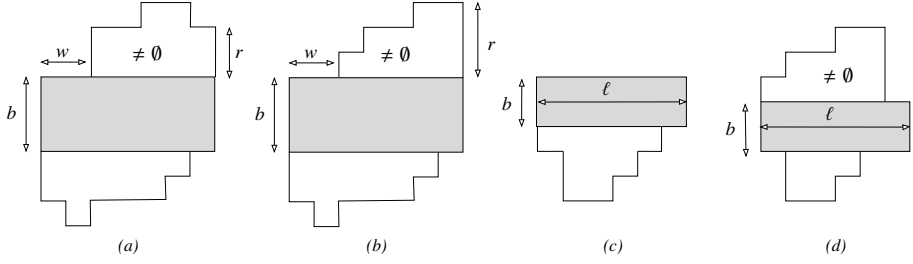}
\caption{Centered ascending polyominoes:
(a) non-rectangular, $b>1$, $r>0$;
(b) rectangular, $b>1$, $r>0$;
(b) flipped stack: $w=\ell(P)$, $r=0$;
(c) non-rectangular, $b>1$, $w=0$, $r=0$.} 
\label{label}
\end{center}
\end{figure}

On the other side, to each non-centered ascending polyomino $P\in{\cal N}$
 we associate a label $(r)$ if $P$ is not rectangular or $(r)'$ if $P$ is rectangular,
with $r\ge1$, where $r$ denotes the number of cells in the last column
(see Figure~\ref{fig:ncstar}). In this case, we will use {\em base} to refer to the set of rows that extend from the left side and have maximal length.  

\begin{comment}
\begin{figure}[htb]
\begin{center}
\scalebox{0.33}{\input{label_nr.pdftex_t}}
\caption{Non-centered ascending polyominoes:
(a) non-rectangular;
(b) rectangular. In both cases, the base is highlighted.}
\label{label_r}
\end{center}
\end{figure}
\end{comment}

\subsection{Operations on centered ascending polyominoes}

We start defining a set of operations on a centered ascending polyomino $P$ of size $n$, 
each producing centered ascending polyominoes of size $n+1$.

\begin{description}
\item[Left Cell.]
This operation can be applied to any centered ascending polyomino.  
It consists of attaching a new cell to the left side of the first column of the base of \(P\), in all admissible vertical positions.

Since the base has height \(b\), this construction produces exactly \(b\) distinct centered ascending polyominoes, each having exactly one cell in the leftmost column.

If the  cell is added in any position \emph{other than the topmost one}, the resulting polyomino satisfies
$
w(P') = 1.
$
Conversely, if the cell is attached to the left of the \emph{topmost cell} of the base, the resulting polyomino \(P'\) satisfies
$w(P') = w(P) + 1$
(see Figure~\ref{fig:leftcell}).

This operation preserves the rectangular property of \(P\).

Moreover, the parameter \(r\) of the resulting polyominoes ranges from \(r\), when the new cell is attached in the topmost position, up to \(r+b-1\), when the cell is attached in the lowest position.

\begin{figure}[htb]
\begin{center}
\scalebox{0.36}{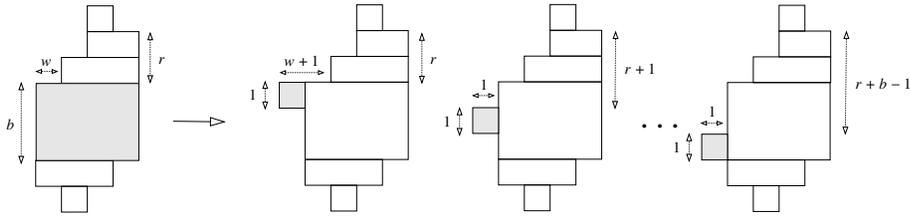}\caption{The application of the operation {\em Left Cell} to a centered ascending polyomino.}\label{fig:leftcell}
\end{center}
\end{figure}
\item[\textbf{Right Cell.}]
This operation can be applied to all centered ascending polyominoes,
except those that have exactly one cell in the leftmost column, in order to avoid ambiguity.

It consists of attaching a new cell to the right side of the base of \(P\),
in all admissible vertical positions.

Since the base has height \(b\), this construction produces exactly \(b\)
distinct centered ascending polyominoes.
For all of them, the resulting polyomino satisfies \(b=1\) and \(r=0\).

If the added cell is placed in any position \emph{other than the topmost one},
the resulting polyomino \(P'\) satisfies
$
w(P') = 0.
$

If instead the cell is attached to the right of the \emph{topmost cell} of the base,
three cases must be distinguished:
\begin{enumerate}
    \item if \(w(P)=0\), then \(w(P')=0\);
    \item if \(P\) is a \emph{flipped stack}, that is, if \(w(P)=\ell(P)\),
    then \(w(P')=w(P)+1\);
    \item in all remaining cases, \(w(P')=w(P)\)
    (see Figure~\ref{fig:rightcell}).
\end{enumerate}

\begin{figure}[htb]
\centering
\scalebox{0.5}{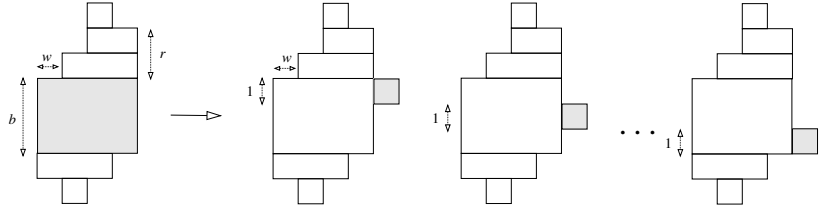}
\caption{Example of the \emph{Right Cell} operation applied to a centered ascending polyomino (case~(3)).}
\label{fig:rightcell}
\end{figure}

Finally, observe that the resulting polyomino \(P'\) is rectangular if and only if
\(P\) is a flipped stack and the added cell is placed in the topmost position.
In all the other cases, the obtained polyomino is not rectangular.

\item[\em Row.]
This operation consists in adding a new row to the base of $P$. 
It can be applied to any centered ascending polyomino, 
and produces a single centered ascending polyomino in which $b$ increases by~1 
while all the other properties remain unchanged (see Figure~\ref{fig:row}).

 \begin{figure}[htb]
\begin{center}
\scalebox{0.5}{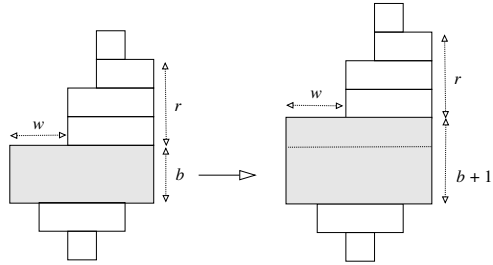}\caption{The application of the operation {\em Row} to a centered ascending polyomino.}\label{fig:row}
\end{center}
\end{figure}

\item[\em Shift.]  
The operation can be applied only if \(w > 0\) and provided that:
\begin{description}
    \item[(i)] \(P\) has more than one cell in the leftmost column; otherwise, the resulting polyomino can be obtained 
    by applying the {\em Left Cell} operation;
    \item[(ii)] \(b = 1\); otherwise, the resulting polyomino can be obtained 
    by applying the {\em Row} operation.
\end{description}

For simplicity, let us distinguish two cases:
\begin{enumerate}
    \item If \(P\) is a flipped stack, (i.e., \(w(P) = \ell(P)\), $r=0$, and $P$ is rectangular)  
    the {\em Shift} operation adds a new row of length ranging from \(1\) to \(\ell(P) - 1\) 
    on top of the uppermost row of the base of \(P\) (see Figure~\ref{fig:shift} (a)).  
    It therefore produces \(w(P) - 1\) rectangular polyominoes, all with \(b = 1\), $r=1$, 
    and with \(w\) taking values from \(\ell(P) - 1\) down to \(1\).
    
    \item Otherwise, the {\em Shift} operation adds a new row of length ranging from \(\ell(P)-w(P)\) to \(\ell(P) - 1\) 
    on top of the uppermost row of the base of \(P\) (see Figure~\ref{fig:shift} (b)).  
    It therefore produces \(w(P)\) polyominoes, all with \(b = 1\), $r$ increased by one, 
    and with \(w\) taking values from \(\ell(P) - 1\) down to \(1\).  In this case, the obtained polyominoes are (all) rectangular if and only if $P$ is rectangular. 
\end{enumerate}

\begin{figure}[htb]
\begin{center}
\scalebox{0.3}{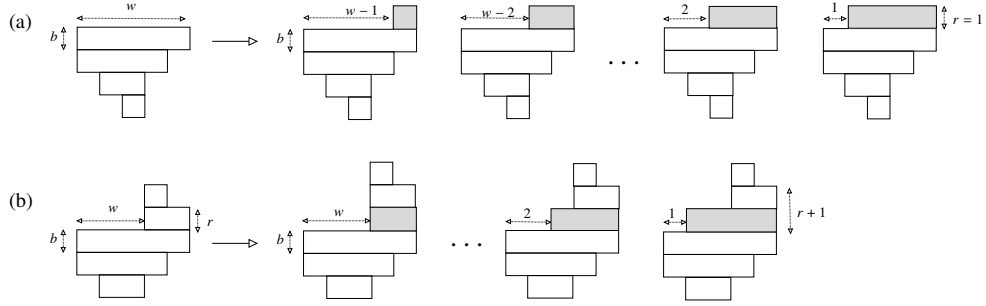}\caption{The application of the operation {\em Shift}: (a)  to a flipped stack (b) to a centered ascending polyomino which is not a flipped stack.}\label{fig:shift}
\end{center}
\end{figure}

\begin{theorem}\label{thm:dec}
Every centered ascending polyomino of size \(n+1\) is uniquely generated from a centered ascending polyomino of size \(n\) through the application of one among the operations {\em Left Cell}, {\em Right Cell}, {\em Row}, or {\em Shift}.
\end{theorem}

\proof
To prove the statement, it suffices to show that for every centered ascending polyomino \(P\) of size \(n+1\), there exists a unique centered ascending polyomino \(\phi(P)\) of size \(n\) such that \(P\) is obtained from \(\phi(P)\) by one of the four operations defined above.  
We proceed by distinguishing the following mutually disjoint cases:

\begin{description}
    \item[(1)] \(b > 1\).  
    In this case, the last applied operation must be {\em Row}, applied to a centered polyomino \(\phi(P)\), since only the {\em Row} operation produces an ascending polyomino satisfying \(b > 1\).

    \item[(2)] \(b = 1\).  
    We consider the following subcases:

    \begin{description}
        \item[(2.1)] \(P\) has exactly one cell in the leftmost column.  
        In this case, \(P\) is obtained by adding a leftmost cell to the centered ascending polyomino \(\phi(P)\) via {\em Left Cell}.

        \item[(2.2)] \(P\) has more than one cell in the leftmost column, but has exactly one cell in the rightmost column.  
        In this case, \(P\) is obtained by adding a rightmost cell to the centered ascending polyomino \(\phi(P)\) via {\em Right Cell}, which is possible because \(\phi(P)\) has more than one cell in the leftmost column.

        \item[(2.3)] Neither case (2.1) nor (2.2) applies.  
        In this case, there is no cell of \(P\) above the first cell of the base, and there is at least one cell below it. Similarly, there is no cell of \(P\) below the last cell of the base, and there is at least one above it.  
        Let \(B_u\) (resp. \(B_d\)) be the row of \(P\) immediately above (resp. below) the base, and let \(\alpha\) be the (possibly empty) stack polyomino lying above \(B_u\).  
        The situation is sketched in Figure~\ref{sketch}.  
        Let \(\phi(P)\) be the polyomino obtained from \(P\) by removing the row \(B_u\).  
        It is a centered ascending polyomino with \(b = 1\). Moreover, \(\phi(P)\) is a flipped stack if and only if \(\alpha\) is empty.  
       It is now clear that \(P\) has been obtained through the application of the {\em Shift} operation to \(\phi(P)\), 
which is possible since \(\phi(P)\) satisfies \(w(\phi(P)) > 0\) and fulfills conditions (1) and (2) required for the application of {\em Shift}.~\qed
    \end{description}
\end{description}

\begin{figure}[htb]
\begin{center}
\scalebox{0.4}{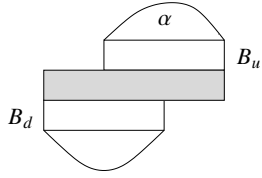}\caption{The situation of case (2.3).}\label{sketch}
\end{center}
\end{figure}
\end{description}

\paragraph{The {\bf \em Nc} operation.}

The \emph{Nc} operation applies to any centered ascending polyomino $P$ with $r>0$.
It generates non-centered polyominoes by attaching a new column of length $r'$,
with $1\le r'\le r$, in all possible positions to the right of the rightmost column
of $P$, strictly above its base (see Figure \ref{fig:nc}).

More precisely, the operation produces $\binom{r+1}{2}$ distinct non-centered
ascending polyominoes. The former base of $P$ becomes the base of the resulting non-centered polyomino.

\begin{figure}[htb]
\begin{center}
\scalebox{0.4}{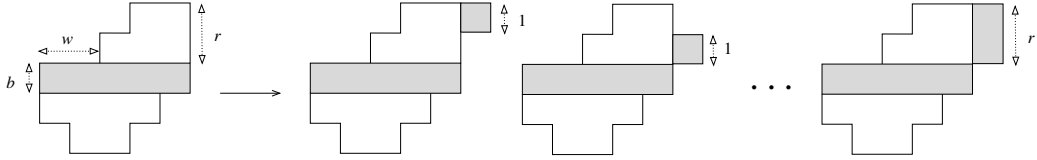}\caption{The application of \emph{Nc} to a rectangular polyomino with $r=2$.}\label{fig:nc}
\end{center}
\end{figure}

If $P$ is rectangular, then the polyomino obtained by adding a column of length $r'$
is rectangular only when  the column is attached at the topmost possible position;
in all other cases, the resulting polyomino is non-rectangular.
Otherwise, if $P$ is not rectangular, then the application of \emph{Nc} always produces
non-rectangular polyominoes.

\subsection{Operations on non-centered ascending polyominoes}

There is a unique operation that can be applied to a non-centered
ascending polyomino of size $n$, producing only non-centered ascending polyominoes of size $n+1$.
This operation mimics the behavior of \emph{Nc}, so it is called \emph{Nc}$^{*}$.
It applies to any non-centered ascending polyomino $P$ having $r>0$
cells in its rightmost column.
It attaches a new column of length $r'$, with $1\le r'\le r$, in all possible
positions to the right of the rightmost column of $P$, strictly above its base
(see Figure~\ref{fig:ncstar}).

Moreover, if $P$ is rectangular, the operation \emph{Nc}$^{*}$ produces one additional
polyomino obtained by adding a single cell on top of the rightmost column of $P$.

Therefore, the \emph{Nc}$^{*}$ operation produces $\binom{r+1}{2}+1$ distinct non-centered
ascending polyominoes if $P$ is rectangular, and $\binom{r+1}{2}$ such polyominoes
otherwise.

\begin{figure}[htb]
\begin{center}
\scalebox{0.33}{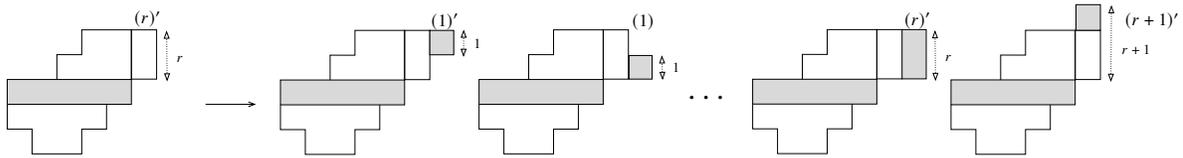}
\caption{The application of \emph{Nc}$^{*}$ to a rectangular non-centered ascending polyomino with $r=2$, together with their labels. The superscript $(r)'$ denotes that the corresponding polyomino is rectangular.}
\label{fig:ncstar}
\end{center}
\end{figure}

If $P$ is rectangular, the polyomino obtained by adding a column of length $r'$
is rectangular only when the column is attached at the topmost possible position,
or when a single cell is added on top of the rightmost column of $P$.
In all other cases, the resulting polyomino is non-rectangular.

If $P$ is not rectangular, then the application of \emph{Nc}$^{*}$ always produces
non-rectangular polyominoes.

\begin{theorem}
Every ascending polyomino of size $n+1$ is uniquely generated from an ascending
polyomino of size $n$ by applying exactly one of the operations
{\em Left Cell}, {\em Right Cell}, {\em Row}, {\em Shift}, {\em Nc}, or {\em Nc}$^*$.
\end{theorem}

\begin{proof}
Let $P$ be an ascending polyomino of size $n+1$.

\medskip
\noindent
\textbf{Case 1: $P$ is non-centered.}

\smallskip
\noindent
(1) Suppose that the topmost cell of the rightmost column of $P$ has strictly
greater height than all other cells of $P$. Then $P$ is rectangular.
By removing this top cell, we obtain a polyomino $\phi(P)$ of size $n$ which is still
non-centered and rectangular. Hence $P$ is obtained from $\phi(P)$ by adding
a cell on top of the last column (operation {\em Nc}$^*$).

\smallskip
\noindent
(2) Otherwise, the topmost cell of the rightmost column has height less than or
equal to the maximal height of the column immediately to its left.
By removing the entire last column, we obtain a polyomino $\phi(P)$ of size $n$.
If $\phi(P)$ is centered, then $P$ is obtained from $\phi(P)$ by applying the
operation {\em Nc}. Otherwise, $P$ is obtained by applying {\em Nc}$^*$.

\medskip
\noindent
\textbf{Case 2: $P$ is centered.}

In this case, the same decomposition argument used in
Theorem~\ref{thm:dec} applies. Therefore, $P$ is uniquely obtained from
$\phi(P)$ by applying exactly one of the operations
{\em Left Cell}, {\em Right Cell}, {\em Row}, or {\em Shift}.

\medskip
\noindent
In all cases, the construction is unique, which concludes the proof.
\end{proof}

\section{A generating tree for ascending polyominoes}

We partition ascending polyominoes into classes according to the values of the parameters 
$b,w,r$ of their elements, so that all polyominoes in a given class are subject to the same set of operations. 
This allows us to formalize the growth of the objects in each class by means of a generating tree.

The case of non-centered ascending polyominoes is the simplest one: as already stated, to a non-centered ascending polyomino we assign the label $(r)'$ if it is rectangular, and $(r)$ otherwise. 

Concerning centered ascending polyominoes, we adopt the convention that, for every class $\chi$, 
the labels of its polyominoes are denoted by $(b,w,r)_\chi$ (non-rectangular) or $(b,w,r)'_\chi$ (rectangular).

\begin{description}

\item[$b>1$:] We further distinguish three subcases (see Figure~\ref{fig:classe_c}):

\begin{description}

\item[\emph{Class $C_0$:}]
These are flipped stacks (always rectangular), characterized by $w>0$ and $r=0$.
Their label is $(b,w,0)'_{C_0}$.

\item[\emph{Class $C$:}]
There is at least one cell above the base and $w>0$; therefore, no cell lies above the top-left corner of the base. 
Their labels are:
\begin{itemize}
\item $(b,w,r)_C$ if the polyomino is non-rectangular, with $r\ge0$;
\item $(b,w,r)'_C$ if the polyomino is rectangular, with $r\ge0$.
\end{itemize}

\item[\emph{Class $C_1$:}]
There is at least one cell above the top-left corner of the base, hence $w=0$ and $r=0$. 
Observe that in this case there are neither cells above nor below the rightmost column of the base, 
so all polyominoes in this class are non-rectangular, with label $(b,0,0)_{C_1}$.

\end{description}

\begin{figure}[ht]
\begin{center}
\scalebox{1}{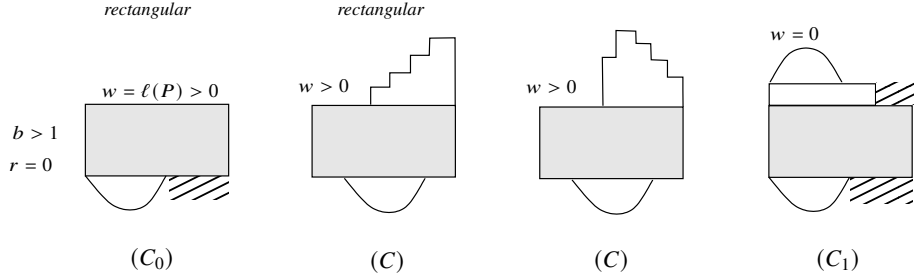}
\caption{The three families with $b>1$. For the family $C$ we show the rectangular and the non-rectangular case.}
\label{fig:classe_c}
\end{center}
\end{figure}

\item[$b=1$:] We further distinguish three subcases:

\begin{description}

\item[\emph{Classes $L_0$ and $L$:}]
The first column of $P$ contains a single cell (hence necessarily $w>0$).
We distinguish two subcases (Figure~\ref{fig:class_l}):
\begin{itemize}
\item class $L_0$, where $P$ is a flipped stack (always rectangular), with label $(1,w,0)'_{L_0}$;
\item class $L$, otherwise, with labels $(1,w,r)_L$ (resp.\ $(1,w,r)'_L$) (rectangular or not).
\end{itemize}

Note that the one-cell polyomino belongs to class $L_0$ and has label $(1,1,0)'_{L_0}$.

\begin{figure}[ht]
\begin{center}
\scalebox{0.8}{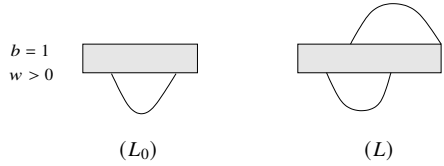}
\caption{The families $L_0$ and $L$.}
\label{fig:class_l}
\end{center}
\end{figure}

\item[\emph{Class $R$:}]
This family consists of centered ascending polyominoes with $w=0$.
This implies that the polyomino has exactly one cell in the rightmost column, hence it is not rectangular and $r=0$.
Their label is $(1,0,0)_R$.

\item[\emph{Classes $S_0$ and $S$:}]
This class comprises all centered ascending polyominoes not considered above, namely those with $b=1$,
having at least one cell below the leftmost cell of the base and no cell above it (thus $w>0$), see Figure~\ref{fig:class_rs}.
We further distinguish:
\begin{itemize}
\item $S_0$: flipped stacks, hence rectangular with $r=0$, label $(1,w,0)'_{S_0}$;
\item $S$: otherwise, with labels $(1,w,r)_S$ (non-rectangular), with $r\geq 0$, and $(1,w,r)'_S$ (rectangular), with $r>0$. Observe that in the non-rectangular case the polyomino may have exactly one cell in the last column, and in this case $r=0$.
\end{itemize}

\end{description}

\begin{figure}[ht]
\begin{center}
\scalebox{0.85}{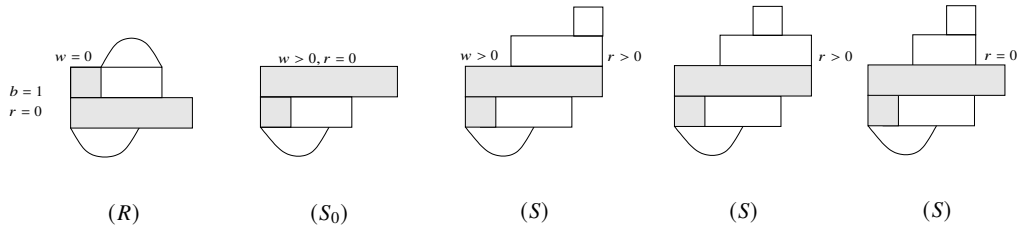}
\caption{The families $R$, $S_0$ and $S$. For the family $S$ we show the cases: rectangular, non-rectangular with $r>0$, and non rectangular with $r=0$.}
\label{fig:class_rs}
\end{center}
\end{figure}

\end{description}

\section{The productions of ascending polyominoes}

In the following we describe the growth of polyominoes of each class, starting from the case of centered ascending polyominoes. The one cell polyomino belongs to $L_0$ and it is rectangular, so it has label $(1,1,0)'_{L_0}$.

\begin{description}
    \item[Classes $C_0$, $C$, $C_1$:] The growths of the three classes are very similar. 
    
    \begin{description}
        \item{- Class $C$:} we can apply all operations for centered ascending polyominoes, except {\em Shift}, (because $b>1$), precisely: {\em LC, RC, Row, Nc}.   
     The productions of a rectangular polyomino in $C$ are shown in Figure \ref{familyC}. Formally:
$$
\begin{array}{llll}
(b,w,r)'_C
&\to 
& (1,w+1,r)'_{L} \, (1,1,r+1)'_L \, \ldots \, (1,1,r+b-1)'_L
& \mbox{\em (LC)} \\
\\
&\to 
& (1,w,0)_S \, (1,0,0)^{b-1}_{R}  
& \mbox{\em (RC)} \\
\\
&\to 
& (b+1,w,r)'_C & \mbox{\em (Row)} \\
\\
&\to 
& (1)' \, (1)^{r-1} \, \ldots \, (r-1)' \, (r-1) \, (r)' & \mbox{\em (Nc)} 
\end{array}
$$

On the other hand, if \(P\) is non--rectangular, the growth is the same,
with the only difference that all the polyominoes produced are non-rectangular.
Consequently, the corresponding production rule is the following:
$$
\begin{array}{llll}
(b,w,r)_C
&\to 
& (1,w+1,r)_{L} \, (1,1,r+1)_L \, \ldots \, (1,1,r+b-1)_L
& \mbox{\em (LC)} \\
\\
&\to 
& (1,w,0)_S \, (1,0,0)^{b-1}_{R}  
& \mbox{\em (RC)} \\
\\
&\to 
& (b+1,w,r)_C & \mbox{\em (Row)} \\
\\
&\to 
& (1)^{r} \, \ldots \, (r-1)^2 \, (r) & \mbox{\em (Nc)} 
\end{array}
$$
    
\begin{figure}[htb]
\begin{center}
 \scalebox{0.55}{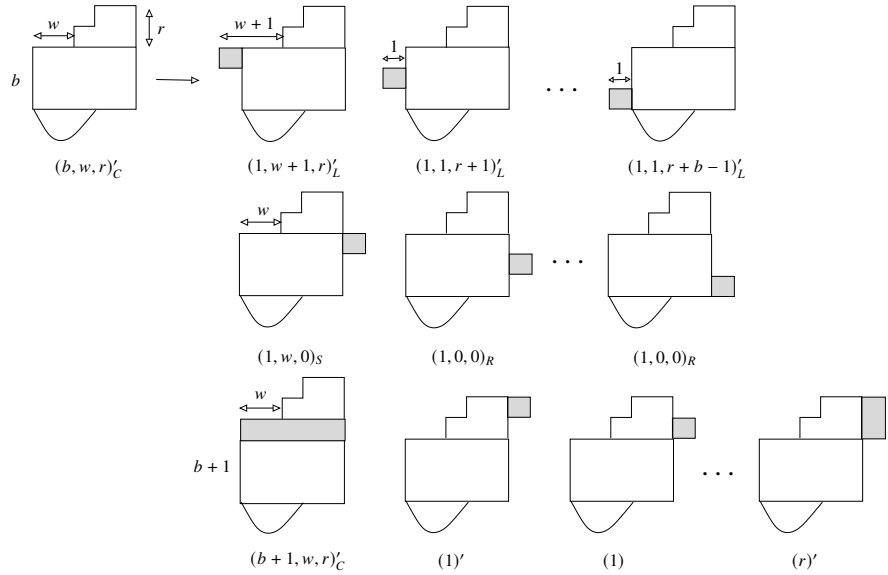}\caption{The growth of a rectangular polyomino in class $C$ .}\label{familyC}
\end{center}
\end{figure}

\item{- Class $C_0$:} The same operations as for class \(C\) can be applied, except {\em Nc}, because $r=0$.  
In this case, however, adding a row, or a cell to the top-left or top-right corner of the base 
results in a polyomino that remains a flipped stack.
Therefore the productions of a polyomino in $C_0$ are:
$$
\begin{array}{llll}
(b,w,0)'_{C_0}
&\to 
& (1,w+1,0)'_{L_0} \, (1,1,1)'_L \, \ldots \, (1,1,b-1)'_L 
 \\
\\
&\to 
& (1,w+1,0)'_{S_0} \, (1,0,0)^{b-1}_{R}  
 \\
\\
&\to 
& (b+1,w,0)'_{C_0} 
\end{array}
$$

\item{- Class $C_1$:} The productions of a polyomino in $C_1$ can be seen as a special case of the productions for the family $C$ (non rectangular case), where $w=r=0$, and precisely, they are:
$$ (b,0,0)_{C_1} \to (1,1,0)_L \ldots (1,1,b-1)_L \,\, (1,0,0)^{b}_{R} \, \, (b+1,0,0)_{C_1} \, .$$

\end{description}
   
     \item[Classes $L$, $L_0$:] To a polyomino in $L$ we can apply only operations {\em LC, Row, Nc} (notice that, as usual, {\em Nc} does not apply if $r=0$), while to a polyomino in $L_0$, being $r=0$, we can apply only operations {\em LC, Row}, giving the productions (see Figure \ref{familyL}): 
$$
\begin{array}{lll} 
(1,w,0)'_{L_0} &\to &(1,w+1,0)'_{L_0} \, \,  (2,w,0)'_{C_0} \\
\\
(1,w,r)'_{L} &\to &(1,w+1,r)'_{L} \, \, (2,w,r)'_{C} \, \, (1)' \, (1)^{r-1} \, \ldots \, (r-1)' \, (r-1) \, (r)' \\
\\
(1,w,r)_{L} &\to &(1,w+1,r)_{L} \, \, (2,w,r)_{C} \, \,  (1)^{r} \, \ldots \, (r-1)^2  \, (r) \\
\end{array}
$$

\begin{figure}[htb]
\begin{center}
 \scalebox{0.6}{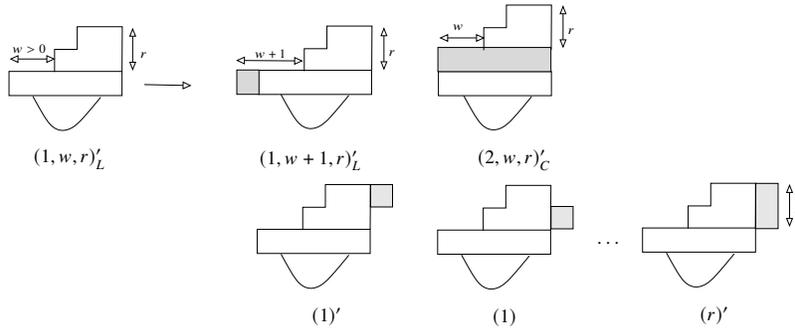}\caption{The growth of a rectangular polyomino in class $L$.}\label{familyL}
\end{center}
\end{figure}
      
      \item[- Class $R$:] 
      In this case, there are no rectangular polyominoes, and only the operations \emph{LC}, \emph{RC}, and \emph{Row} can be applied, since $r=w=0$ (see Figure~\ref{familyR}).
    The production is:

      $$
      (1,0,0)_{R} \to (1,1,0)_{L} \, \, (1,1,0)_{R} \, \, (2,0,0)_{C_1}
      $$

\begin{figure}[htb]
\begin{center}
\scalebox{0.65}{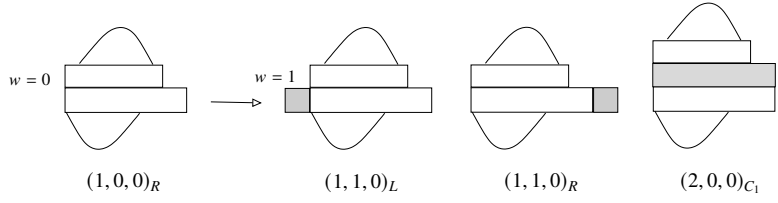}\caption{The growth of a polyomino in  $R$.}\label{familyR}
\end{center}
\end{figure}
      
       \item[- Classes $S$, $S_0$:] To a polyomino in $S$ (resp. $S_0$) we can apply all operations (resp. all operations except {\em Nc}, because $r=0$), see Figure \ref{familyS}. The main difference consists in the fact that the application of {\em RC} produces a polyomino $P'$ such that:
       \begin{description}
           \item{- $w(P')=w(P)+1$}, if $P$ is in $S_0$ (flipped stack);
            \item{- $w(P')=w(P)$}, if $P$ is in $S$,
       \end{description}
       and we should distinguish among the productions of polyominoes in $S$ which are rectangular, and those which are not.  
       The productions are:

   $$
\begin{array}{lll}
(1,w,0)'_{S_0}
&\to & (1,w+1,0)'_{L_0} \, (1,w+1,0)'_{S_0} \, (2,w,0)'_{C_0} \\
& & (1,1,1)'_{S} \, (1,2,1)'_S \, \ldots \, (1,w-1,1)'_S \, ; \\
\\
(1,w,r)'_{S}
&\to 
& (1,w+1,r)'_{L} \, (1,w,0)_{S} \, (2,w,r)'_{C} \\
& &(1,1,r+1)'_{S} (1,2,r+1)'_S \ldots (1,w,r+1)'_S \\
& &(1)' \, \ldots \, (r)' \, \, (1)^{r-1} \, \ldots \, (r-1)  \\ 
\\
(1,w,r)_{S}
&\to 
& (1,w+1,r)_{L} \, (1,w,0)_{S} \, (2,w,r)_{C} \\
& &(1,1,r+1)_{S} (1,2,r+1)_S \ldots (1,w,r+1)_S \\
& &(1)^{r} \, \ldots \, (r)  \\ 
\end{array}
$$

\begin{figure}[htb]
\begin{center}
\scalebox{0.55}{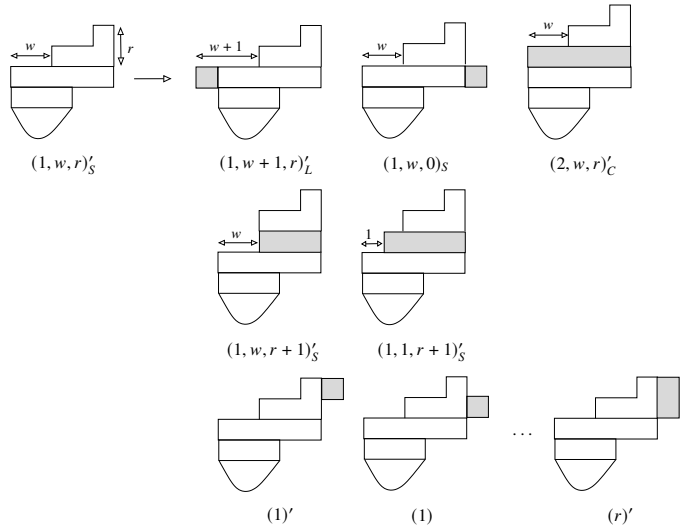}\caption{The growth of a rectangular polyomino in  $S$.}\label{familyS}
\end{center}
\end{figure}

Finally, each non-centered ascending polyomino in $\mathcal{N}$ is labeled by $(r)$ or $(r)'$ according to whether it is non-rectangular or rectangular, respectively.
In this case, the only applicable operation is \emph{Nc}$^{*}$ (see Figure~\ref{fig:ncstar}):
 $$
\begin{array}{lll}
(r) &\to &(1)^r \, \ldots \, (r-1)^2 \, (r) \\
\\
(r)' &\to &(1)' \, \ldots \, (r)' \, (r+1)' \\
& &(1)^{r-1} \, (2)^{r-2} \ldots \,  (r-1) \\
\end{array}
$$

\end{description}

Collecting all productions, the full set of generating tree productions for ascending polyominoes is:
{\scriptsize $$
\left\{\begin{array}{lll}
(1,1,0)'_{L_0} & & \\
\\
(b,w,r)'_C
&\to 
& (1,w+1,r)'_{L} \, (1,1,r+1)'_L \, \ldots \, (1,1,r+b-1)'_L
\\
\\
&\to 
& (1,w,0)_S \, (1,0,0)^{b-1}_{R}  \\
\\
&\to 
& (b+1,w,r)'_C \\
\\
&\to 
& (1)' \, (1)^{r-1} \, \ldots \, (r-1)' \, (r-1) \, (r)'  \\
\\
(b,w,r)_C
&\to 
& (1,w+1,r)_{L} \, (1,1,r+1)_L \, \ldots \, (1,1,r+b-1)_L
\\
\\
&\to 
& (1,w,0)_S \, (1,0,0)^{b-1}_{R}  
 \\
\\
&\to 
& (b+1,w,r)_C  \\
\\
&\to 
& (1)^{r} \, \ldots \, (r-1)^2 \, (r)  \\
\\
(b,w,0)'_{C_0}
&\to 
& (1,w+1,0)'_{L_0} \, (1,1,1)'_L \, \ldots \, (1,1,b-1)'_L 
 \\
\\
&\to 
& (1,w+1,0)'_{S_0} \, (1,0,0)^{b-1}_{R}  
 \\
\\
&\to 
& (b+1,w,0)'_{C_0} \\
\\
(1,w,0)'_{L_0} &\to &(1,w+1,0)'_{L_0} \, \,  (2,w,0)'_{C_0} \\
\\
(1,w,r)'_{L} &\to &(1,w+1,r)'_{L} \, \, (2,w,r)'_{C} \, \, (1)' \, (1)^{r-1} \, \ldots \, (r-1)' \, (r-1) \, (r)' \\
\\
(1,w,r)_{L} &\to &(1,w+1,r)_{L} \, \, (2,w,r)_{C} \, \,  (1)^{r} \, \ldots \, (r-1)^2  \, (r) \\
\\
(1,0,0)_{R} &\to &(1,1,0)_{L} \, \, (1,1,0)_{R} \, \, (2,0,0)_{C_1} \\
\\
(1,w,0)'_{S_0}
&\to & (1,w+1,0)'_{L_0} \, (1,w+1,0)'_{S_0} \, (2,w,0)'_{C_0} \\
& & (1,1,1)'_{S} \, (1,2,1)'_S \, \ldots \, (1,w-1,1)'_S \, ; \\
\\
(1,w,r)'_{S}
&\to 
& (1,w+1,r)'_{L} \, (1,w,0)_{S} \, (2,w,r)'_{C} \\
& &(1,1,r+1)'_{S} (1,2,r+1)'_S \ldots (1,w,r+1)'_S \\
& &(1)' \, \ldots \, (r)' \, \, (1)^{r-1} \, \ldots \, (r-1)  \\ 
\\
(1,w,r)_{S}
&\to 
& (1,w+1,r)_{L} \, (1,w,0)_{S} \, (2,w,r)_{C} \\
& &(1,1,r+1)_{S} (1,2,r+1)_S \ldots (1,w,r+1)_S \\
& &(1)^{r} \, \ldots \, (r)  \\ 
(r) &\to &(1)^r \, \ldots \, (r-1)^2 \, (r) \\
\\
(r)' &\to &(1)' \, \ldots \, (r)' \, (r+1)' \\
& &(1)^{r-1} \, (2)^{r-2} \ldots \,  (r-1) \\
\end{array} \right.
$$}

\section{Solving the functional equations}

Using standard techniques~\cite{gfgt,eco,mbm2,convex}, the generating tree can be translated into a system of functional equations satisfied by the generating functions of the classes of centered and non-centered ascending polyominoes.
Concerning the centered cases, for any given class $K$, we denote by: \[K(x,y,z;t) =K(x,y,z)= \sum_{P \in K} t^{|P|} x^{b(P)} y^{w(P)} z^{r(P)}\]
the generating function of $K$ according to the parameters $b,w,r$ and the size $|P|$ of $P$.
\paragraph{Rectangular Case} From the succession rule above we find:
\begin{align}
    C'_0(x,y) &= tx C'_0(x,y) + tx L'_0(x,y) + tx S'_0(x,y) \label{eqn:C0p} \\[2ex]
    L'_0(x,y) &= t^2xy + txyC'_0(1,y) + tyL'_0(x,y) + tyS'_0(x,y) \label{eqn:L0p} \\[2ex]
    S'_0(x,y) &= txy C'_0(1,y) + ty S'_0(x,y) \label{eqn:S0p} \\[2ex]
    C'(x,y,z) &= tx C'(x,y,z) + tx L'(x,y,z) + tx S'(x,y,z) \label{eqn:Cp} \\[2ex]
    L'(x,y,z) &= \begin{multlined}[t] txy C'(1,y,z) + \frac{txy}{1-z}\left[z C'(1,1,z)-C'(z,1,z)\right] + \frac{txy}{1-z}\left[zC'_0(1,1)-C'_0(z,1)\right] \\ + ty L'(x,y,z) + ty S'(x,y,z) \end{multlined} \label{eqn:Lp} \\[2ex]
    S'(x,y,z) &= \frac{tz}{1-y}\left[yS'_0(x,1)-S'_0(x,y)\right] + \frac{tyz}{1-y}\left[S'(x,1,z)-S'(x,y,z)\right] \label{eqn:Sp} \\[2ex]
    N'(z) &= \begin{multlined}[t] \frac{tz}{1-z}\left[C'(1,1,1) - C'(1,1,z)\right] + \frac{tz}{1-z}\left[L'(1,1,1)-L'(1,1,z)\right] \\ +\frac{tz}{1-z}\left[S'(1,1,1)-S'(1,1,z)\right] + \frac{tz}{1-z}\left[N'(1) - zN'(z)\right] \end{multlined} \label{eqn:Np}
\end{align}
Equations \eqref{eqn:C0p}--\eqref{eqn:S0p} can be solved on their own, giving the (rational) known generating functions of families of stack polyominoes:
\begin{align}
    C'_0(x,y) &= \frac{ t^3 x^2 y(1-t)(1-ty)}{(1-tx) (1 - t - 2 t y + t^2 y^2)}\,, \\
    L'_0(x,y) &= \frac{t^2 x y (1-ty - t)}{1 - t - 2 t y + t^2 y^2}\,, \\
    S'_0(x,y) &= \frac{t^4 x y^2}{1 - t - 2 t y + t^2 y^2}\,.
\end{align}
These solutions can then be substituted into \eqref{eqn:Sp}, leaving $S'(x,y,z)$ and $S'(x,1,z)$ as the only unknowns in that equation. After rearranging, the standard kernel method can be applied, using the kernel root $y=\frac{1}{1-tz}$. This gives:
\begin{equation}
    S'(x,y,z) = \frac{t^5 x y z (1 - t - t^2 y - t z + t^2 z)}{(1 - t - 2 t y + t^2 y^2) (1 - 3 t + t^2 - 2 t z + 4 t^2 z + t^2 z^2 - t^3 z^2)}\,.
\end{equation}
Next, \eqref{eqn:Cp} and \eqref{eqn:Lp} can be combined, and substituting $x=y=1$ or $x=z$ and $y=1$ gives enough equations to solve for $C'$ and $L'$:
\begin{align}
    C'(x,y,z) &= \frac{t^5 x^2 y z (1-t)(1 - t y + t^2 y - t z - t^2 z + t^2 y z)}{(1-tx) (1 - t - 2 t y + t^2 y^2)(1 - 3 t + t^2 - 2 t z + 4 t^2 z + t^2 z^2 - t^3 z^2)}\,, \\
    L'(x,y,z) &= \frac{t^4xyz(1 - 2 t + t^2 - t y + 2 t^2 y - t z + t^2 z + t^2 y z - t^3 y z)}{(1 - t - 2 t y + t^2 y^2)(1 - 3 t + t^2 - 2 t z + 4 t^2 z + t^2 z^2 - t^3 z^2)}\,.
\end{align}
Finally we can substitute everything into \eqref{eqn:Np}, leaving $N'(z)$ and $N'(1)$. The kernel method can be applied again, with $z=\frac{1-\sqrt{1-4t}}{2t}$, giving the solution:
\begin{equation}
    N'(z) = \frac{t^3z}{\sqrt{1-4t} (1 - z + t z^2)} - \frac{t^3 z (1 - t z) (1 - t - t z)}{(1 - z + t z^2)(1 - 3 t + t^2 - 2 t z + 4 t^2 z + t^2 z^2 - t^3 z^2)}\,.
\end{equation}

\paragraph{Non-rectangular case} For brevity we will use the notation:
\[ \partial_{x=1}f(x) = \left[\frac{\partial}{\partial x}f(x)\right]_{x=1} \]
and likewise for $\partial_{z=1}$.
From the succession rules, we have
\begin{align}
    C(x,y,z) &= tx C(x,y,z) + tx L(x,y,z) + tx S(x,y,z), \label{eqn:C} \\[2ex]
    C_1(x) &= tx C_1(x) + tx R(x), \label{eqn:C1} \\[2ex]
    L(x,y,z) &= \begin{multlined}[t] txy C(1,y,z) + \frac{txy}{1-z}\left[zC(1,1,z) - C(z,1,z)\right] + \frac{txy}{1-z}\left[C_1(1) - C_1(z)\right] \\ + ty L(x,y,z) + ty R(x) + ty S(x,y,z), \end{multlined} \label{eqn:L} \\[2ex]
    R(x) &= \begin{multlined}[t] tx \partial_{x=1}C'(x,1,1) - tx C'(1,1,1) +tx \partial_{x=1}C(x,1,1) - tx C(1,1,1) \\ + tx \partial_{x=1}C'_0(x,1) - tx C'_0(1,1) + tx \partial_{x=1}C_1(x) + tR(x), \end{multlined} \label{eqn:R} \\[2ex]
    S(x,y,z) &= t x C'(1,y,1) + tx C(1,y,1) + tS'(x,y,1) + t S(x,y,1) + \frac{tyz}{1-y}\left[S(x,1,z) - S(x,y,z)\right], \label{eqn:S}
\end{align}
and one final equation for $N(z)$ which we will deal with separately.

Note that $L(x,y,z) = xL(1,y,z)$ and similarly $R(x) = xR(1)$ and $S(x,y,z) = xS(1,y,z)$, while $C$ and $C_1$ do not factor so easily. So $x$ is not really required in \eqref{eqn:L}--\eqref{eqn:S}, and we can write:
\begin{align}
    L(1,y,z) &= \begin{multlined}[t] ty C(1,y,z) + \frac{ty}{1-z}\left[zC(1,1,z) - C(z,1,z)\right] + \frac{ty}{1-z}\left[C_1(1) - C_1(z)\right] \\ + ty L(1,y,z) + ty R(1) + ty S(1,y,z), \end{multlined} \label{eqn:L_x1} \\[2ex]
    R(1) &= \begin{multlined}[t] t \partial_{x=1}C'(x,1,1) - t C'(1,1,1) +t \partial_{x=1}C(x,1,1) - t C(1,1,1) \\ + t \partial_{x=1}C'_0(x,1) - t C'_0(1,1) + t \partial_{x=1}C_1(x) + tR(1),\end{multlined} \label{eqn:R_x1} \\[2ex]
    S(1,y,z) &= t C'(1,y,1) + t C(1,y,1) + tS'(1,y,1) + t S(1,y,1) + \frac{tyz}{1-y}\left[S(1,1,z) - S(1,y,z)\right]. \label{eqn:S_x1}
\end{align}
Using the above we find that:
\begin{align}
    C(x,y,z) &= \frac{x^2(1-t)}{1-tx}C(1,y,z), \\
    C_1(x) &= \frac{x^2(1-t)}{1-tx}C_1(1) = \frac{tx^2}{1-tx}R(1).
\end{align}
We then get:
\begin{equation}
    \partial_{x=1}C_1(x) = \frac{t(2-t)}{(1-t)^2}R(1)\, ,
\end{equation}
\begin{equation}
    \partial_{x=1}C(x,y,z) = \frac{2-t}{1-t}C(1,y,z).
\end{equation}
Similar results also apply for $C'$ and $C'_0$.
With these simplifications we no longer need $x$ at all, and after eliminating $C_1(1)$ we get the system:
\begin{align}
    C(1,y,z) &= \frac{t}{1-t}L(1,y,z) + \frac{t}{1-t} S(1,y,z), \label{eqn:C_x1} \\[2ex]
    L(1,y,z) &= \begin{multlined}[t] \frac{tyz}{(1-ty)(1-tz)}C(1,1,z) + \frac{ty}{1-ty}C(1,y,z) + \frac{ty}{(1-t)(1-ty)(1-tz)}R(1) \\ + \frac{ty}{1-ty}S(1,y,z), \end{multlined} \label{eqn:L_x1_subs} \\[2ex]
    R(1) &= \frac{t(1-t)}{1-3t+t^2}C(1,1,1) + \frac{t(1-t)}{1-3t+t^2}C'_0(1,1) + \frac{t(1-t)}{1-3t+t^2}C'(1,1,1), \label{eqn:R_x1_subs} \\[2ex]
    S(1,y,z) &= \begin{multlined}[t] \frac{t(1-y)}{1-y+tyz}C(1,y,1) + \frac{t(1-y)}{1-y+tyz}C'(1,y,1) + \frac{ty}{1-y+tyz}S(1,1,z) \\ +\frac{t(1-y)}{1-y+tyz}S(1,y,1) + \frac{t(1-y)}{1-y+tyz}S'(1,y,1).\end{multlined} \label{eqn:S_x1_subs}
\end{align}
Equations \eqref{eqn:C_x1}--\eqref{eqn:S_x1_subs} can be combined to eventually eliminate all unknowns except $S(1,y,z)$, $S(1,y,1)$ and $S(1,1,1)$. The resulting equation looks like:
\begin{multline}
    (1 - y + t y^2) (1 - y + t y z)S(1,y,z) - t y (1 - y + t y^2) z S(1,1,z) \\ - \frac{t^2 (1-y) y (1 - 3 t + t^2 - t y + 2 t^2 y) }{1 - 5 t + 6 t^2 - t^3}S(1,1,1) = \text{RHS}, \label{eqn:all_x1_subs}
\end{multline}
where the RHS depends only on $C'_0(1,1)$, $C'(1,1,1)$, $C'(1,y,1)$ and $S'(1,y,1)$ (which are all known).

We apply the kernel method twice to \eqref{eqn:all_x1_subs}, first setting $y=\frac{1}{1-tz}$ which eliminates the $S(1,y,z)$ term, and then with $z=\frac{1-\sqrt{1-4t}}{2t}$ which eliminates the $S(1,1,z)$ term. This gives an algebraic solution to $S(1,1,1)$, which can then be substituted back to give $S(1,1,z)$. (We do not need the $y$ variable when solving for $N$.) This is a fairly large expression, but at $z=1$ we have:
\begin{equation}
    S(1,1,1) = \frac{t (1 - 5 t + 6 t^2 - t^3)}{2 (1 - 2 t)\sqrt{1-4t}} - \frac{t (1 - 8 t + 23 t^2 - 28 t^3 + 14 t^4 - 4 t^5 + t^6)}{2 (1 - 2 t) (1 - 5 t + 6 t^2 - t^3)}\,.
\end{equation}
We can then set $y=z=1$ in \eqref{eqn:C_x1}--\eqref{eqn:R_x1_subs} to get three equations in order to solve $R(1)$, and hence $C_1(1)$:
\begin{align}
    R(1) &= \frac{t^3(1-t)}{2(1-2t)\sqrt{1-4t}} - \frac{t^3(1-t)}{2(1-2t)}\,, \\
    C_1(1) &= \frac{t^4}{2(1-2t)\sqrt{1-4t}} - \frac{t^4}{2(1-2t)}\,.
\end{align}

Then we go back again and set $y=1$ to get two equations which give solutions to $C(1,1,z)$ and $L(1,1,z)$. Again we get fairly large expressions, but at $z=1$
\begin{align}
    C(1,1,1) &= \frac{t^2 (1 - 3 t + t^2)}{2 (1 - 2 t) \sqrt{1-4t}} - \frac{t^2 (1 - 6 t + 12 t^2 - 8 t^3 + t^4 - t^5)}{2 (1 - 2 t) (1 - 5 t + 6 t^2 - t^3)}\,, \\
    L(1,1,1) &= \frac{t^2}{2\sqrt{1-4t}} - \frac{t^2 (1 - 3 t + 2 t^2 - t^3)}{2 (1 - 5 t + 6 t^2 - t^3)}\,.
\end{align}

Finally we turn to $N(z)$. Let us define the operators $\Delta_1$ and $\Delta_2$ by:
\begin{align}
    \Delta_1 f(z) &= \frac{tz}{1-z} \partial_{z=1}f(z) - \frac{tz}{(1-z)^2}\left[f(1)-f(z)\right], \\
    \Delta_2 f(z) &= \frac{tz}{1-z} \partial_{z=1}f(z) - \frac{tz^2}{(1-z)^2}\left[f(1)-f(z)\right].
\end{align}
Then the succession rule gives:
\begin{multline}
    N(z) = \Delta_1 C'(1,1,z) + \Delta_2 C(1,1,z) + \Delta_1 L'(1,1,z) + \Delta_2 L(1,1,z)  \\ + \Delta_1 S'(1,1,z) + \Delta_2 S(1,1,z) + \Delta_1 N'(z) + \Delta_2 N(z). \label{eqn:N}
\end{multline}
The only unknowns here are $N(z)$, $N(1)$ and $\partial_{z=1}N(z)$.

Combining the $N(z)$ terms gives the kernel $1-\frac{tz^2}{(1-z)^2}$, which has two Puiseux series as roots:
\begin{align}
    Z_\pm &= \frac{1 \pm \sqrt{t}}{1-t} \\
    &= 1 \pm t^{1/2} + t \pm t^{3/2} + t^2 \pm t^{5/2} + t^3 \pm t^{7/2} + t^4 \pm t^{9/2} + t^5 \pm t^{11/2} + t^6 + \dots
\end{align}
These can both be substituted into \eqref{eqn:N} to give two equations in $N(1)$ and $\partial_{z=1}N(z)$, which fortunately turn out to be linearly independent. Combining them gives the solution:
\begin{equation}
    N(1) = \frac{t (1 - 10 t + 31 t^2 - 16 t^3 - 68 t^4 + 90 t^5 - 27 t^6 + 4 t^7)}{2 (1 - 4 t)^2 (1 - 5 t + 6 t^2 - t^3)} - \frac{t (1 - 5 t + 2 t^2 + 13 t^3 - 8 t^4)}{2(1-2t)(1-4t)^{3/2}}\,.
\end{equation}

This completes the solution of the generating functions. 

\subsection{Generating functions of some families of ascending polyominoes}

From the generating functions obtained above we can compute generating functions of several families of ascending polyominoes, starting from centered ascending ones. 

\begin{theorem}
 For $n\geq 0$, the number $h(n)$ of centered ascending polyominoes is:
 \begin{equation}\label{eq:an}
     % h(n+2)=\frac{3n+1}{n+1}\binom{2n-1}{n} \, .
     h(n) = \frac{3n-5}{n-1}\binom{2n-5}{n-2}\, .
 \end{equation}
\end{theorem}

\proof 
Adding up everything except $N$ and $N'$ gives the centered ascending polyominoes:
\begin{multline}
    H(t) = \frac{t(1-t)}{2\sqrt{1-4t}} - \frac{t(1-t)}{2} \\ = t^2 + 2 t^3 + 7 t^4 + 25 t^5 + 91 t^6 + 336 t^7 + 1254 t^8 + 
 4719 t^9 + 17875 t^{10} + \dots
\end{multline}
Hence we have the closed formula. \qed

This is sequence A097613 in the OEIS \cite{oeis}.

\begin{comment}
Recall that the Catalan number is defined by
\[
C(n) = \frac{1}{n+1}\binom{2n}{n}.
\]

\begin{corollary}\label{prop:cat}
For $n\geq 0$, we have:
\[
 \; e(n+2) = \frac{3n+1}{2}\,C({n}) \; = 3 \, \binom{2n-1}{n} - C(n).
\]
\end{corollary}
The coefficients here are
\begin{equation}
    c(n) = \begin{cases} 
    1, & n=2 \\
    \frac12\left(\binom{2n-2}{n-1} - \binom{2n-4}{n-2}\right), & n \geq 3 \end{cases}
\end{equation}
Alternatively from the OEIS (sequence A097613) we can get the formula
\begin{equation}
    c(n) = \binom{2n-5}{n-2} + \binom{2n-4}{n-3}, \qquad n \geq 2.
\end{equation}
\end{comment}

\begin{theorem}
The number of rectangular polyominoes of size $n$ is $\binom{2n-4}{n-2}$.    
\end{theorem}
\proof Adding up all the rectangular cases gives:
\begin{equation}
    \frac{t^2}{\sqrt{1-4t}}  = t^2 + 2 t^3 + 6 t^4 + 20 t^5 + 70 t^6 + 252 t^7 + 924 t^8 + 
 3432 t^9 + 12870t^{10} + \cdots
\end{equation}
This has a simple explanation once we observe that rectangular polyominoes are exactly directed-convex polyominoes up to a rotation of $180^\circ$ counterclockwise. According to \cite{mbm2,DV,convex}, the statement follows. \qed

Adding up everything gives all ascending polyominoes:

\begin{theorem}
The generating function of ascending polyominoes is:
\begin{multline}
    A(t) = \frac{t^2 (2 - 12 t + 19 t^2 - 4 t^3)}{2(1-4t)^2} - \frac{t^4(5-8t)}{2(1-2t)(1-4t)^{3/2}} \\
    = t^2 + 2 t^3 + 7 t^4 + 26 t^5 + 101 t^6 + 404 t^7 + 1649 t^8 + 
 6824 t^9 + 28498t^{10} + \dots
\end{multline}
\end{theorem}

A (not so nice) closed formula for the number $a(n)$ can be derived from the generating function. Even without this, we can show that the number of ascending polyominoes grows asymptotically as: 
\[ a(n) \sim
4^n \, 
\frac{n}{256} 
\]
From the generating function we observe that the dominant asymptotic growth
$a(n) \sim K \,n\,4^n$ is determined by the rational part of the generating function
(the left-hand term), whose double pole at $x=\tfrac14$ produces the $n4^n$ factor.
The algebraic part contributes only lower-order terms and, in particular,
yields a negative correction to the leading asymptotics. The same phenomenon occurs for convex polyominoes (see Introduction), where the rational part is responsible for the dominant asymptotic behaviour.

\subsection{The number of $Z$-convex polyominoes (reprise)}

With the above results, we can refine the enumeration of $Z$-convex polyominoes. 

\paragraph{The class $\mcC(2,2)$} We recall from Section \ref{sec:definitions} that convex polyominoes $\cal C$ are obtained as the union:
\[ {\cal C}= \mathcal{A} \cup \mathcal{D} \cup \mcC(2,2) \, ,\]
and that the intersection of $\mathcal{A}$ and $\mathcal{D}$ gives $L$-convex polyominoes. Hence, the generating function $C_{2,2}(t)$ of $\mcC(2,2)$ is obtained as:
\[ C_{2,2}(t)=C(t)-2A(t)+L(t),\]
where $C(t)$ (resp. $A(t)$, $L(t)$) is the generating function of convex (resp. ascending, $L$-convex) polyominoes. 
Therefore we have: 
\begin{align}
C_{2,2}(t)&= \frac{t^4}{\left(1-2 t\right) \left(1-4 t\right)^{\frac{3}{2}}} - \frac{t^4}{\left(1-4 t\right) \left(2 t^{2}-4 t+1\right)} \\
&= 2 t^{8}+32 t^{9}+308 t^{10}+2320 t^{11}+15094 t^{12}+89104 t^{13}+491012 t^{14} + \ldots \notag 
\end{align}
Hence, the asymptotic growth of polyominoes in $\mcC(2,2)$ is: 
\[
c_{2,2}(n) \sim \frac{1}{64\sqrt{\pi}}\sqrt{n} 4^n \, ,
\]
which is the rough asymptotic behaviour of $4$-stacks (we recall that $\mcC(2,2)$ is included in $4$-stacks).

\paragraph{The class $\mcC(2,1)$} We recall from Section \ref{sec:definitions} that $Z$-convex polyominoes are the disjoint union:
\begin{center}
$Z$-convex $= \mcC(2,1) \cup \mcC(1,2) \cup \mcC(2,2) \cup \mbox{$L$-convex}$
\end{center}
Therefore the generating function $C_{2,1}(t)$ of $\mcC(2,1)$ is obtained as:
\[ C_{2,1}(t)= \frac{1}{2} \, \left( Z(t) - C_{2,2}(t) - L(t) \right) ,\]
Precisely, 
{\small
\begin{align}
C_{2,1}(t) &=  \frac{\left(8 t^{3}-15 t^{2}+10 t-2\right) t^{4}}{2 \left(1-t\right) \left(1-3 t\right) \left(1-4 t\right)^{\frac{3}{2}} \left(1-2 t\right)} - \frac{  t^4 \left(8 t^{5}-6 t^{4}+4 t^{3}-25 t^{2}+14 t-2\right) }{2 \left(2 t^{2}-4 t+1\right)  \left(1-4 t\right)^2 \left(1-3 t\right) \left(1-t\right)}
\\
&= 2 t^{5}+17 t^{6}+102 t^{7}+532 t^{8}+2576 t^{9}+11919 t^{10}+53504 t^{11}+235115 t^{12}+1017218 t^{13} + \ldots \notag 
\end{align}}
% As in the case of ascending polyominoes, the algebraic part (the second term of the sum) brings a negative contribution. 

In particular, we observe that the sequence counting polyominoes in $\mcC(2,1)$
has asymptotic growth of the form $c_{2,1}(n) \sim \frac{1}{768}n4^n$.
We can now complete Table~\ref{tab1} in the Introduction, with the rough asymptotic behaviour of the sequences of the classes studied in the paper:

\begin{table}[htb]
\begin{center}
\begin{tabular}{l|c|c|c}
  {\bf Class of polyominoes} &{\bf Type of g.f.} & \textbf{OEIS entry} &{\bf Asymptotic growth} \\
  \hline
   & & \\
 {\bf Centered ascending} & algebraic & \href{https://oeis.org/A097613}{A097613} &  $\frac{3}{32} \;  \frac{4^n}{\sqrt{\pi n}}$ \\
   & & \\
 {\bf Ascending} & algebraic & New &  $\frac{n}{256} \;  4^n$ \\
   & & \\
 {\bf $\mcC(2,2)$} & algebraic & New & $\frac{1}{64\sqrt{\pi}}\sqrt{n}\; 4^n$ \\
  & & \\
 {\bf $\mcC(2,1)$} & algebraic & New & $\frac{n}{768} \; 4^n$ \\
  & & \\
  {\bf $Z$-convex} & algebraic & \href{https://oeis.org/A128611}{A128611} & $\frac{n}{384}\;  4^n$ \\
\end{tabular}
\end{center}
\caption{The asymptotic behavior of the classes studied in the paper. The quantities in the last column correspond to the number of polyominoes of size $n$.}
\end{table}

\paragraph{Final Remarks.}

Building on known enumerative results, Marc Noy in 2005 (personal communication) pointed out the problem of identifying a class of convex polyominoes lying between the families of $L$-convex and $Z$-convex polyominoes, whose generating function is rational and whose asymptotic growth is $n4^n$.
Our results tell us that the family we are seeking must be contained in the set
\[
\mcC(2,1)\setminus \{\text{4-stacks}\}.
\]
or equivalently, in $
\mcC(1,2)\setminus \{\text{4-stacks}\}.
$

\section*{Acknowledgements}

The authors would like to thank Enrica Duchi and Gilles Schaeffer 
for their valuable advice and support during the long development of this work.

%\bibitem{noy}
%M. Noy, Some comments on 2-convex polyominoes, Personal communication.

\end{document}

%% file: c221.pdftex_t
\begin{picture}(0,0)%
\includegraphics{c221.pdf}%
\end{picture}%
%
%  Created by WinFIG version 2024.2 
%  METADATA <version>1.0</version> 
%
\setlength{\unitlength}{3947sp}%
\begin{picture}(7246,2708)(934,-4542)
%  METADATA <id>34</id> 
\put(7876,-3729){\makebox(0,0)[lb]{\smash{\fontsize{12}{14.4}\usefont{T1}{ptm}{m}{n}{\color[rgb]{0,0,0}$\gamma$}%
}}}
%  METADATA <id>11</id> 
\put(2102,-2176){\makebox(0,0)[lb]{\smash{\fontsize{12}{14.4}\usefont{T1}{ptm}{m}{n}{\color[rgb]{0,0,0}$\delta$}%
}}}
%  METADATA <id>9</id> 
\put(3489,-2806){\makebox(0,0)[lb]{\smash{\fontsize{12}{14.4}\usefont{T1}{ptm}{m}{n}{\color[rgb]{0,0,0}$\beta$}%
}}}
%  METADATA <id>8</id> 
\put(1134,-3699){\makebox(0,0)[lb]{\smash{\fontsize{12}{14.4}\usefont{T1}{ptm}{m}{n}{\color[rgb]{0,0,0}$\alpha$}%
}}}
%  METADATA <id>10</id> 
\put(3009,-4343){\makebox(0,0)[lb]{\smash{\fontsize{12}{14.4}\usefont{T1}{ptm}{m}{n}{\color[rgb]{0,0,0}$\gamma$}%
}}}
%  METADATA <id>28</id> 
\put(6901,-2206){\makebox(0,0)[lb]{\smash{\fontsize{12}{14.4}\usefont{T1}{ptm}{m}{n}{\color[rgb]{0,0,0}$\beta$}%
}}}
%  METADATA <id>25</id> 
\put(5514,-2784){\makebox(0,0)[lb]{\smash{\fontsize{12}{14.4}\usefont{T1}{ptm}{m}{n}{\color[rgb]{0,0,0}$\delta$}%
}}}
%  METADATA <id>31</id> 
\put(5994,-4351){\makebox(0,0)[lb]{\smash{\fontsize{12}{14.4}\usefont{T1}{ptm}{m}{n}{\color[rgb]{0,0,0}$\alpha$}%
}}}
\end{picture}%

%% file: c22.pdftex_t
\begin{picture}(0,0)%
\includegraphics{c22.pdf}%
\end{picture}%
%
%  Created by WinFIG version 2024.2 
%  METADATA <version>1.0</version> 
%
\setlength{\unitlength}{3947sp}%
\begin{picture}(7984,7504)(740,-6928)
%  METADATA <id>129</id> 
\put(1554,-4629){\makebox(0,0)[lb]{\smash{\fontsize{18}{21.6}\usefont{T1}{ptm}{m}{n}{\color[rgb]{0,0,0}$\alpha$}%
}}}
%  METADATA <id>133</id> 
\put(4786,-474){\makebox(0,0)[lb]{\smash{\fontsize{18}{21.6}\usefont{T1}{ptm}{m}{n}{\color[rgb]{0,0,0}$r_\delta$}%
}}}
%  METADATA <id>135</id> 
\put(6174,-4194){\makebox(0,0)[lb]{\smash{\fontsize{18}{21.6}\usefont{T1}{ptm}{m}{n}{\color[rgb]{0,0,0}$r_\alpha$}%
}}}
%  METADATA <id>136</id> 
\put(5109,-6106){\makebox(0,0)[lb]{\smash{\fontsize{18}{21.6}\usefont{T1}{ptm}{m}{n}{\color[rgb]{0,0,0}$r_\gamma$}%
}}}
%  METADATA <id>137</id> 
\put(4839,-1329){\makebox(0,0)[lb]{\smash{\fontsize{18}{21.6}\usefont{T1}{ptm}{m}{n}{\color[rgb]{0,0,0}$r_\epsilon$}%
}}}
%  METADATA <id>138</id> 
\put(4801,-5214){\makebox(0,0)[lb]{\smash{\fontsize{18}{21.6}\usefont{T1}{ptm}{m}{n}{\color[rgb]{0,0,0}$r_\zeta$}%
}}}
%  METADATA <id>139</id> 
\put(1576,-3601){\makebox(0,0)[lb]{\smash{\fontsize{18}{21.6}\usefont{T1}{ptm}{m}{n}{\color[rgb]{0,0,0}$c_\alpha$}%
}}}
%  METADATA <id>140</id> 
\put(7726,-3121){\makebox(0,0)[lb]{\smash{\fontsize{18}{21.6}\usefont{T1}{ptm}{m}{n}{\color[rgb]{0,0,0}$c_\beta$}%
}}}
%  METADATA <id>130</id> 
\put(7756,-1937){\makebox(0,0)[lb]{\smash{\fontsize{18}{21.6}\usefont{T1}{ptm}{m}{n}{\color[rgb]{0,0,0}$\beta$}%
}}}
%  METADATA <id>131</id> 
\put(3496,-474){\makebox(0,0)[lb]{\smash{\fontsize{18}{21.6}\usefont{T1}{ptm}{m}{n}{\color[rgb]{0,0,0}$\delta$}%
}}}
%  METADATA <id>132</id> 
\put(6301,-6144){\makebox(0,0)[lb]{\smash{\fontsize{18}{21.6}\usefont{T1}{ptm}{m}{n}{\color[rgb]{0,0,0}$\gamma$}%
}}}
%  METADATA <id>134</id> 
\put(2754,-2357){\makebox(0,0)[lb]{\smash{\fontsize{18}{21.6}\usefont{T1}{ptm}{m}{n}{\color[rgb]{0,0,0}$r_\beta$}%
}}}
%  METADATA <id>141</id> 
\put(2896,-3354){\makebox(0,0)[lb]{\smash{\fontsize{18}{21.6}\usefont{T1}{ptm}{m}{n}{\color[rgb]{0,0,0}$\cal R$}%
}}}
\end{picture}%

%% file: label_r.pdftex_t
\begin{picture}(0,0)%
\includegraphics{label_r.pdf}%
\end{picture}%
%
%  Created by WinFIG version 2024.2 
%  METADATA <version>1.0</version> 
%
\setlength{\unitlength}{3947sp}%
\begin{picture}(17224,4864)(742,-4958)
%  METADATA <id>79</id> 
\put(12100,-1865){\makebox(0,0)[lb]{\smash{\fontsize{22}{26.4}\usefont{T1}{ptm}{m}{n}{\color[rgb]{0,0,0}$\ell$}%
}}}
%  METADATA <id>9</id> 
\put(757,-2326){\makebox(0,0)[lb]{\smash{\fontsize{22}{26.4}\usefont{T1}{ptm}{m}{n}{\color[rgb]{0,0,0}$b$}%
}}}
%  METADATA <id>13</id> 
\put(1702,-1145){\makebox(0,0)[lb]{\smash{\fontsize{22}{26.4}\usefont{T1}{ptm}{m}{n}{\color[rgb]{0,0,0}$w$}%
}}}
%  METADATA <id>34</id> 
\put(3151,-1201){\makebox(0,0)[lb]{\smash{\fontsize{24}{28.8}\usefont{T1}{ptm}{m}{it}{\color[rgb]{0,0,0}$\neq \emptyset$}%
}}}
%  METADATA <id>40</id> 
\put(5122,-1167){\makebox(0,0)[lb]{\smash{\fontsize{22}{26.4}\usefont{T1}{ptm}{m}{n}{\color[rgb]{0,0,0}$r$}%
}}}
%  METADATA <id>54</id> 
\put(2889,-4832){\makebox(0,0)[lb]{\smash{\fontsize{20}{24}\usefont{T1}{ptm}{m}{it}{\color[rgb]{0,0,0}(a)}%
}}}
%  METADATA <id>62</id> 
\put(5482,-2341){\makebox(0,0)[lb]{\smash{\fontsize{22}{26.4}\usefont{T1}{ptm}{m}{n}{\color[rgb]{0,0,0}$b$}%
}}}
%  METADATA <id>64</id> 
\put(6427,-1160){\makebox(0,0)[lb]{\smash{\fontsize{22}{26.4}\usefont{T1}{ptm}{m}{n}{\color[rgb]{0,0,0}$w$}%
}}}
%  METADATA <id>65</id> 
\put(7876,-1216){\makebox(0,0)[lb]{\smash{\fontsize{24}{28.8}\usefont{T1}{ptm}{m}{it}{\color[rgb]{0,0,0}$\neq \emptyset$}%
}}}
%  METADATA <id>67</id> 
\put(9922,-1003){\makebox(0,0)[lb]{\smash{\fontsize{22}{26.4}\usefont{T1}{ptm}{m}{n}{\color[rgb]{0,0,0}$r$}%
}}}
%  METADATA <id>68</id> 
\put(7614,-4847){\makebox(0,0)[lb]{\smash{\fontsize{20}{24}\usefont{T1}{ptm}{m}{it}{\color[rgb]{0,0,0}(b)}%
}}}
%  METADATA <id>35</id> 
\put(16306,-1628){\makebox(0,0)[lb]{\smash{\fontsize{24}{28.8}\usefont{T1}{ptm}{m}{it}{\color[rgb]{0,0,0}$\neq \emptyset$}%
}}}
%  METADATA <id>28</id> 
\put(14440,-2630){\makebox(0,0)[lb]{\smash{\fontsize{22}{26.4}\usefont{T1}{ptm}{m}{n}{\color[rgb]{0,0,0}$b$}%
}}}
%  METADATA <id>30</id> 
\put(16352,-2338){\makebox(0,0)[lb]{\smash{\fontsize{22}{26.4}\usefont{T1}{ptm}{m}{n}{\color[rgb]{0,0,0}$\ell$}%
}}}
%  METADATA <id>70</id> 
\put(16357,-4846){\makebox(0,0)[lb]{\smash{\fontsize{20}{24}\usefont{T1}{ptm}{m}{it}{\color[rgb]{0,0,0}(d)}%
}}}
%  METADATA <id>80</id> 
\put(12123,-4845){\makebox(0,0)[lb]{\smash{\fontsize{20}{24}\usefont{T1}{ptm}{m}{it}{\color[rgb]{0,0,0}(c)}%
}}}
%  METADATA <id>77</id> 
\put(10188,-2157){\makebox(0,0)[lb]{\smash{\fontsize{22}{26.4}\usefont{T1}{ptm}{m}{n}{\color[rgb]{0,0,0}$b$}%
}}}
\end{picture}%

%% file: lcell_r.pdftex_t
\begin{picture}(0,0)%
\includegraphics{lcell_r.pdf}%
\end{picture}%
%
%  Created by WinFIG version 2024.2 
%  METADATA <version>1.0</version> 
%
\setlength{\unitlength}{3947sp}%
\begin{picture}(14847,3647)(896,-2858)
%  METADATA <id>99</id> 
\put(3571,-251){\makebox(0,0)[lb]{\smash{\fontsize{16}{19.2}\usefont{T1}{ptm}{m}{n}{\color[rgb]{0,0,0}$r$}%
}}}
%  METADATA <id>92</id> 
\put(5146,-891){\makebox(0,0)[lb]{\smash{\fontsize{16}{19.2}\usefont{T1}{ptm}{m}{n}{\color[rgb]{0,0,0}$1$}%
}}}
%  METADATA <id>77</id> 
\put(9178,-702){\makebox(0,0)[lb]{\smash{\fontsize{16}{19.2}\usefont{T1}{ptm}{m}{n}{\color[rgb]{0,0,0}$1$}%
}}}
%  METADATA <id>80</id> 
\put(8538,-1352){\makebox(0,0)[lb]{\smash{\fontsize{16}{19.2}\usefont{T1}{ptm}{m}{n}{\color[rgb]{0,0,0}$1$}%
}}}
%  METADATA <id>86</id> 
\put(13163,-1184){\makebox(0,0)[lb]{\smash{\fontsize{16}{19.2}\usefont{T1}{ptm}{m}{n}{\color[rgb]{0,0,0}$1$}%
}}}
%  METADATA <id>89</id> 
\put(12543,-1804){\makebox(0,0)[lb]{\smash{\fontsize{16}{19.2}\usefont{T1}{ptm}{m}{n}{\color[rgb]{0,0,0}$1$}%
}}}
%  METADATA <id>110</id> 
\put(15728,-699){\makebox(0,0)[lb]{\smash{\fontsize{16}{19.2}\usefont{T1}{ptm}{m}{n}{\color[rgb]{0,0,0}$r+b-1$}%
}}}
%  METADATA <id>58</id> 
\put(911,-1411){\makebox(0,0)[lb]{\smash{\fontsize{16}{19.2}\usefont{T1}{ptm}{m}{n}{\color[rgb]{0,0,0}$b$}%
}}}
%  METADATA <id>59</id> 
\put(1511,-291){\makebox(0,0)[lb]{\smash{\fontsize{16}{19.2}\usefont{T1}{ptm}{m}{n}{\color[rgb]{0,0,0}$w$}%
}}}
%  METADATA <id>104</id> 
\put(8311,-261){\makebox(0,0)[lb]{\smash{\fontsize{16}{19.2}\usefont{T1}{ptm}{m}{n}{\color[rgb]{0,0,0}$r$}%
}}}
%  METADATA <id>107</id> 
\put(11701,-441){\makebox(0,0)[lb]{\smash{\fontsize{16}{19.2}\usefont{T1}{ptm}{m}{n}{\color[rgb]{0,0,0}$r+1$}%
}}}
%  METADATA <id>75</id> 
\put(5756,-301){\makebox(0,0)[lb]{\smash{\fontsize{16}{19.2}\usefont{T1}{ptm}{m}{n}{\color[rgb]{0,0,0}$w+1$}%
}}}
%  METADATA <id>112</id> 
\put(11960,-1364){\makebox(0,0)[lb]{\smash{\fontsize{36}{43.2}\usefont{T1}{ptm}{m}{n}{\color[rgb]{0,0,0}$\ldots$}%
}}}
\end{picture}%

%% file: rcell_r.pdftex_t
\begin{picture}(0,0)%
\includegraphics{rcell_r.pdf}%
\end{picture}%
%
%  Created by WinFIG version 2024.2 
%  METADATA <version>1.0</version> 
%
\setlength{\unitlength}{2763sp}%
\begin{picture}(14544,3636)(896,-2870)
%  METADATA <id>80</id> 
\put(8513,-1401){\makebox(0,0)[lb]{\smash{\fontsize{11}{13.2}\usefont{T1}{ptm}{m}{n}{\color[rgb]{0,0,0}$1$}%
}}}
%  METADATA <id>107</id> 
\put(5181,-916){\makebox(0,0)[lb]{\smash{\fontsize{11}{13.2}\usefont{T1}{ptm}{m}{n}{\color[rgb]{0,0,0}$1$}%
}}}
%  METADATA <id>89</id> 
\put(12686,-1846){\makebox(0,0)[lb]{\smash{\fontsize{11}{13.2}\usefont{T1}{ptm}{m}{n}{\color[rgb]{0,0,0}$1$}%
}}}
%  METADATA <id>58</id> 
\put(911,-1411){\makebox(0,0)[lb]{\smash{\fontsize{11}{13.2}\usefont{T1}{ptm}{m}{n}{\color[rgb]{0,0,0}$b$}%
}}}
%  METADATA <id>59</id> 
\put(1511,-291){\makebox(0,0)[lb]{\smash{\fontsize{11}{13.2}\usefont{T1}{ptm}{m}{n}{\color[rgb]{0,0,0}$w$}%
}}}
%  METADATA <id>111</id> 
\put(3631,-241){\makebox(0,0)[lb]{\smash{\fontsize{11}{13.2}\usefont{T1}{ptm}{m}{n}{\color[rgb]{0,0,0}$r$}%
}}}
%  METADATA <id>104</id> 
\put(5814,-311){\makebox(0,0)[lb]{\smash{\fontsize{11}{13.2}\usefont{T1}{ptm}{m}{n}{\color[rgb]{0,0,0}$w$}%
}}}
%  METADATA <id>112</id> 
\put(11851,-1591){\makebox(0,0)[lb]{\smash{\fontsize{25}{30}\usefont{T1}{ptm}{m}{n}{\color[rgb]{0,0,0}$\ldots$}%
}}}
\end{picture}%

%% file: rowcentered_r.pdftex_t
\begin{picture}(0,0)%
\includegraphics{rowcentered_r.pdf}%
\end{picture}%
%
%  Created by WinFIG version 2024.2 
%  METADATA <version>1.0</version> 
%
\setlength{\unitlength}{3158sp}%
\begin{picture}(7034,4074)(1329,-3345)
%  METADATA <id>40</id> 
\put(3941,-2051){\makebox(0,0)[lb]{\smash{\fontsize{13}{15.6}\usefont{T1}{ptm}{m}{n}{\color[rgb]{0,0,0}$b$}%
}}}
%  METADATA <id>39</id> 
\put(6058,-772){\makebox(0,0)[lb]{\smash{\fontsize{13}{15.6}\usefont{T1}{ptm}{m}{n}{\color[rgb]{0,0,0}$w$}%
}}}
%  METADATA <id>41</id> 
\put(8348,-2082){\makebox(0,0)[lb]{\smash{\fontsize{13}{15.6}\usefont{T1}{ptm}{m}{n}{\color[rgb]{0,0,0}$b+1$}%
}}}
%  METADATA <id>46</id> 
\put(3934,-978){\makebox(0,0)[lb]{\smash{\fontsize{13}{15.6}\usefont{T1}{ptm}{m}{n}{\color[rgb]{0,0,0}$r$}%
}}}
%  METADATA <id>50</id> 
\put(8321,-543){\makebox(0,0)[lb]{\smash{\fontsize{13}{15.6}\usefont{T1}{ptm}{m}{n}{\color[rgb]{0,0,0}$r$}%
}}}
%  METADATA <id>38</id> 
\put(1701,-1261){\makebox(0,0)[lb]{\smash{\fontsize{13}{15.6}\usefont{T1}{ptm}{m}{n}{\color[rgb]{0,0,0}$w$}%
}}}
\end{picture}%

%% file: shift_op_r.pdftex_t
\begin{picture}(0,0)%
\includegraphics{shift_op_r.pdf}%
\end{picture}%
%
%  Created by WinFIG version 2024.2 
%  METADATA <version>1.0</version> 
%
\setlength{\unitlength}{3947sp}%
\begin{picture}(19789,6340)(458,-9598)
%  METADATA <id>168</id> 
\put(12988,-4648){\makebox(0,0)[lb]{\smash{\fontsize{36}{43.2}\usefont{T1}{ptm}{m}{n}{\color[rgb]{0,0,0}$\ldots$}%
}}}
%  METADATA <id>169</id> 
\put(9701,-8448){\makebox(0,0)[lb]{\smash{\fontsize{36}{43.2}\usefont{T1}{ptm}{m}{n}{\color[rgb]{0,0,0}$\ldots$}%
}}}
%  METADATA <id>83</id> 
\put(473,-3885){\makebox(0,0)[lb]{\smash{\fontsize{26}{31.2}\usefont{T1}{ptm}{m}{n}{\color[rgb]{0,0,0}(a)}%
}}}
%  METADATA <id>109</id> 
\put(6074,-4198){\makebox(0,0)[lb]{\smash{\fontsize{18}{21.6}\usefont{T1}{ptm}{m}{n}{\color[rgb]{0,0,0}$b$}%
}}}
%  METADATA <id>110</id> 
\put(7516,-3616){\makebox(0,0)[lb]{\smash{\fontsize{18}{21.6}\usefont{T1}{ptm}{m}{n}{\color[rgb]{0,0,0}$w-1$}%
}}}
%  METADATA <id>119</id> 
\put(10366,-3661){\makebox(0,0)[lb]{\smash{\fontsize{18}{21.6}\usefont{T1}{ptm}{m}{n}{\color[rgb]{0,0,0}$w-2$}%
}}}
%  METADATA <id>20</id> 
\put(6070,-8422){\makebox(0,0)[lb]{\smash{\fontsize{18}{21.6}\usefont{T1}{ptm}{m}{n}{\color[rgb]{0,0,0}$b$}%
}}}
%  METADATA <id>21</id> 
\put(7277,-7832){\makebox(0,0)[lb]{\smash{\fontsize{18}{21.6}\usefont{T1}{ptm}{m}{n}{\color[rgb]{0,0,0}$w$}%
}}}
%  METADATA <id>82</id> 
\put(473,-7644){\makebox(0,0)[lb]{\smash{\fontsize{26}{31.2}\usefont{T1}{ptm}{m}{n}{\color[rgb]{0,0,0}(b)}%
}}}
%  METADATA <id>8</id> 
\put(1324,-8400){\makebox(0,0)[lb]{\smash{\fontsize{18}{21.6}\usefont{T1}{ptm}{m}{n}{\color[rgb]{0,0,0}$b$}%
}}}
%  METADATA <id>9</id> 
\put(2552,-7739){\makebox(0,0)[lb]{\smash{\fontsize{18}{21.6}\usefont{T1}{ptm}{m}{n}{\color[rgb]{0,0,0}$w$}%
}}}
%  METADATA <id>162</id> 
\put(4624,-8040){\makebox(0,0)[lb]{\smash{\fontsize{18}{21.6}\usefont{T1}{ptm}{m}{n}{\color[rgb]{0,0,0}$r$}%
}}}
%  METADATA <id>156</id> 
\put(17684,-3556){\makebox(0,0)[lb]{\smash{\fontsize{18}{21.6}\usefont{T1}{ptm}{m}{n}{\color[rgb]{0,0,0}$1$}%
}}}
%  METADATA <id>128</id> 
\put(14549,-3586){\makebox(0,0)[lb]{\smash{\fontsize{18}{21.6}\usefont{T1}{ptm}{m}{n}{\color[rgb]{0,0,0}$2$}%
}}}
%  METADATA <id>160</id> 
\put(20232,-3808){\makebox(0,0)[lb]{\smash{\fontsize{18}{21.6}\usefont{T1}{ptm}{m}{n}{\color[rgb]{0,0,0}$r=1$}%
}}}
%  METADATA <id>31</id> 
\put(11252,-7863){\makebox(0,0)[lb]{\smash{\fontsize{18}{21.6}\usefont{T1}{ptm}{m}{n}{\color[rgb]{0,0,0}$2$}%
}}}
%  METADATA <id>51</id> 
\put(14276,-7863){\makebox(0,0)[lb]{\smash{\fontsize{18}{21.6}\usefont{T1}{ptm}{m}{n}{\color[rgb]{0,0,0}$1$}%
}}}
%  METADATA <id>165</id> 
\put(16976,-7845){\makebox(0,0)[lb]{\smash{\fontsize{18}{21.6}\usefont{T1}{ptm}{m}{n}{\color[rgb]{0,0,0}$r+1$}%
}}}
%  METADATA <id>92</id> 
\put(1331,-4154){\makebox(0,0)[lb]{\smash{\fontsize{18}{21.6}\usefont{T1}{ptm}{m}{n}{\color[rgb]{0,0,0}$b$}%
}}}
%  METADATA <id>93</id> 
\put(2930,-3507){\makebox(0,0)[lb]{\smash{\fontsize{18}{21.6}\usefont{T1}{ptm}{m}{n}{\color[rgb]{0,0,0}$w$}%
}}}
\end{picture}%

%% file: sketch.pdftex_t
\begin{picture}(0,0)%
\includegraphics{sketch.pdf}%
\end{picture}%
%
%  Created by WinFIG version 2024.2 
%  METADATA <version>1.0</version> 
%
\setlength{\unitlength}{3947sp}%
\begin{picture}(3621,2661)(836,-3111)
%  METADATA <id>9</id> 
\put(3214,-814){\makebox(0,0)[lb]{\smash{\fontsize{22}{26.4}\usefont{T1}{ptm}{m}{n}{\color[rgb]{0,0,0}$\alpha$}%
}}}
%  METADATA <id>7</id> 
\put(4442,-1429){\makebox(0,0)[lb]{\smash{\fontsize{22}{26.4}\usefont{T1}{ptm}{m}{n}{\color[rgb]{0,0,0}$B_u$}%
}}}
%  METADATA <id>8</id> 
\put(851,-2374){\makebox(0,0)[lb]{\smash{\fontsize{22}{26.4}\usefont{T1}{ptm}{m}{n}{\color[rgb]{0,0,0}$B_d$}%
}}}
\end{picture}%

%% file: nc.pdftex_t
\begin{picture}(0,0)%
\includegraphics{nc.pdf}%
\end{picture}%
%
%  Created by WinFIG version 2024.2 
%  METADATA <version>1.0</version> 
%
\setlength{\unitlength}{3947sp}%
\begin{picture}(16075,2406)(876,-3917)
%  METADATA <id>9</id> 
\put(891,-2861){\makebox(0,0)[lb]{\smash{\fontsize{16}{19.2}\usefont{T1}{ptm}{m}{n}{\color[rgb]{0,0,0}$b$}%
}}}
%  METADATA <id>13</id> 
\put(1761,-2171){\makebox(0,0)[lb]{\smash{\fontsize{16}{19.2}\usefont{T1}{ptm}{m}{n}{\color[rgb]{0,0,0}$w$}%
}}}
%  METADATA <id>46</id> 
\put(4161,-2121){\makebox(0,0)[lb]{\smash{\fontsize{16}{19.2}\usefont{T1}{ptm}{m}{n}{\color[rgb]{0,0,0}$r$}%
}}}
%  METADATA <id>77</id> 
\put(16936,-2146){\makebox(0,0)[lb]{\smash{\fontsize{16}{19.2}\usefont{T1}{ptm}{m}{n}{\color[rgb]{0,0,0}$r$}%
}}}
%  METADATA <id>57</id> 
\put(8931,-1881){\makebox(0,0)[lb]{\smash{\fontsize{16}{19.2}\usefont{T1}{ptm}{m}{n}{\color[rgb]{0,0,0}$1$}%
}}}
%  METADATA <id>67</id> 
\put(12161,-2321){\makebox(0,0)[lb]{\smash{\fontsize{16}{19.2}\usefont{T1}{ptm}{m}{n}{\color[rgb]{0,0,0}$1$}%
}}}
%  METADATA <id>86</id> 
\put(12556,-2881){\makebox(0,0)[lb]{\smash{\fontsize{36}{43.2}\usefont{T1}{ptm}{m}{n}{\color[rgb]{0,0,0}$\ldots$}%
}}}
\end{picture}%

%% file: ncstar.pdftex_t
\begin{picture}(0,0)%
\includegraphics{ncstar.pdf}%
\end{picture}%
%
%  Created by WinFIG version 2024.2 
%  METADATA <version>1.0</version> 
%
\setlength{\unitlength}{3947sp}%
\begin{picture}(21293,2876)(466,-3907)
%  METADATA <id>125</id> 
\put(21694,-1981){\makebox(0,0)[lb]{\smash{\fontsize{16}{19.2}\usefont{T1}{ptm}{m}{n}{\color[rgb]{0,0,0}$r+1$}%
}}}
%  METADATA <id>77</id> 
\put(3711,-2131){\makebox(0,0)[lb]{\smash{\fontsize{16}{19.2}\usefont{T1}{ptm}{m}{n}{\color[rgb]{0,0,0}$r$}%
}}}
%  METADATA <id>109</id> 
\put(18336,-2123){\makebox(0,0)[lb]{\smash{\fontsize{16}{19.2}\usefont{T1}{ptm}{m}{n}{\color[rgb]{0,0,0}$r$}%
}}}
%  METADATA <id>132</id> 
\put(2888,-1398){\makebox(0,0)[lb]{\smash{\fontsize{20}{24}\usefont{T1}{ptm}{m}{n}{\color[rgb]{0,0,0}$(r)'$}%
}}}
%  METADATA <id>133</id> 
\put(8538,-1448){\makebox(0,0)[lb]{\smash{\fontsize{20}{24}\usefont{T1}{ptm}{m}{n}{\color[rgb]{0,0,0}$(1)'$}%
}}}
%  METADATA <id>134</id> 
\put(12363,-1460){\makebox(0,0)[lb]{\smash{\fontsize{20}{24}\usefont{T1}{ptm}{m}{n}{\color[rgb]{0,0,0}$(1)$}%
}}}
%  METADATA <id>93</id> 
\put(9363,-1873){\makebox(0,0)[lb]{\smash{\fontsize{16}{19.2}\usefont{T1}{ptm}{m}{n}{\color[rgb]{0,0,0}$1$}%
}}}
%  METADATA <id>138</id> 
\put(13426,-2836){\makebox(0,0)[lb]{\smash{\fontsize{36}{43.2}\usefont{T1}{ptm}{m}{n}{\color[rgb]{0,0,0}$\ldots$}%
}}}
%  METADATA <id>135</id> 
\put(17544,-1430){\makebox(0,0)[lb]{\smash{\fontsize{20}{24}\usefont{T1}{ptm}{m}{n}{\color[rgb]{0,0,0}$(r)'$}%
}}}
%  METADATA <id>136</id> 
\put(21744,-1443){\makebox(0,0)[lb]{\smash{\fontsize{20}{24}\usefont{T1}{ptm}{m}{n}{\color[rgb]{0,0,0}$(r+1)'$}%
}}}
%  METADATA <id>101</id> 
\put(13141,-2326){\makebox(0,0)[lb]{\smash{\fontsize{16}{19.2}\usefont{T1}{ptm}{m}{n}{\color[rgb]{0,0,0}$1$}%
}}}
\end{picture}%

%% file: classeC.pdftex_t
\begin{picture}(0,0)%
\includegraphics{classeC.pdf}%
\end{picture}%
%
%  Created by WinFIG version 2024.2 
%  METADATA <version>1.0</version> 
%
\setlength{\unitlength}{3947sp}%
\begin{picture}(5715,1780)(1156,-1656)
%  METADATA <id>120</id> 
\put(1766, -6){\makebox(0,0)[lb]{\smash{\fontsize{7}{8.4}\usefont{T1}{ptm}{m}{it}{\color[rgb]{0,0,0}rectangular}%
}}}
%  METADATA <id>68</id> 
\put(1926,-1581){\makebox(0,0)[lb]{\smash{\fontsize{9}{10.8}\usefont{T1}{ptm}{m}{it}{\color[rgb]{0,0,0}$(C_0)$}%
}}}
%  METADATA <id>119</id> 
\put(3231, -6){\makebox(0,0)[lb]{\smash{\fontsize{7}{8.4}\usefont{T1}{ptm}{m}{it}{\color[rgb]{0,0,0}rectangular}%
}}}
%  METADATA <id>71</id> 
\put(6234,-1586){\makebox(0,0)[lb]{\smash{\fontsize{9}{10.8}\usefont{T1}{ptm}{m}{it}{\color[rgb]{0,0,0}$(C_1)$}%
}}}
%  METADATA <id>105</id> 
\put(5949,-166){\makebox(0,0)[lb]{\smash{\fontsize{7}{8.4}\usefont{T1}{ptm}{m}{it}{\color[rgb]{0,0,0}$w=0$}%
}}}
%  METADATA <id>111</id> 
\put(4838,-1592){\makebox(0,0)[lb]{\smash{\fontsize{9}{10.8}\usefont{T1}{ptm}{m}{it}{\color[rgb]{0,0,0}$(C)$}%
}}}
%  METADATA <id>114</id> 
\put(4416,-501){\makebox(0,0)[lb]{\smash{\fontsize{7}{8.4}\usefont{T1}{ptm}{m}{it}{\color[rgb]{0,0,0}$w>0$}%
}}}
%  METADATA <id>56</id> 
\put(3431,-1596){\makebox(0,0)[lb]{\smash{\fontsize{9}{10.8}\usefont{T1}{ptm}{m}{it}{\color[rgb]{0,0,0}$(C)$}%
}}}
%  METADATA <id>104</id> 
\put(2986,-491){\makebox(0,0)[lb]{\smash{\fontsize{7}{8.4}\usefont{T1}{ptm}{m}{it}{\color[rgb]{0,0,0}$w>0$}%
}}}
%  METADATA <id>93</id> 
\put(1191,-791){\makebox(0,0)[lb]{\smash{\fontsize{7}{8.4}\usefont{T1}{ptm}{m}{it}{\color[rgb]{0,0,0}$b>1$}%
}}}
%  METADATA <id>103</id> 
\put(1751,-526){\makebox(0,0)[lb]{\smash{\fontsize{7}{8.4}\usefont{T1}{ptm}{m}{it}{\color[rgb]{0,0,0}$w=\ell(P)>0$}%
}}}
%  METADATA <id>108</id> 
\put(1171,-991){\makebox(0,0)[lb]{\smash{\fontsize{7}{8.4}\usefont{T1}{ptm}{m}{it}{\color[rgb]{0,0,0}$r=0$}%
}}}
\end{picture}%

%% file: classe_L.pdftex_t
\begin{picture}(0,0)%
\includegraphics{classe_L.pdf}%
\end{picture}%
%
%  Created by WinFIG version 2024.2 
%  METADATA <version>1.0</version> 
%
\setlength{\unitlength}{3947sp}%
\begin{picture}(3419,1276)(607,-1300)
%  METADATA <id>90</id> 
\put(622,-635){\makebox(0,0)[lb]{\smash{\fontsize{8}{9.6}\usefont{T1}{ptm}{m}{it}{\color[rgb]{0,0,0}$w>0$}%
}}}
%  METADATA <id>37</id> 
\put(3404,-1235){\makebox(0,0)[lb]{\smash{\fontsize{10}{12}\usefont{T1}{ptm}{m}{it}{\color[rgb]{0,0,0}$(L)$}%
}}}
%  METADATA <id>38</id> 
\put(632,-449){\makebox(0,0)[lb]{\smash{\fontsize{8}{9.6}\usefont{T1}{ptm}{m}{it}{\color[rgb]{0,0,0}$b=1$}%
}}}
%  METADATA <id>41</id> 
\put(1480,-1230){\makebox(0,0)[lb]{\smash{\fontsize{10}{12}\usefont{T1}{ptm}{m}{it}{\color[rgb]{0,0,0}$(L_0)$}%
}}}
\end{picture}%

%% file: classe_RS.pdftex_t
\begin{picture}(0,0)%
\includegraphics{classe_RS.pdf}%
\end{picture}%
%
%  Created by WinFIG version 2024.2 
%  METADATA <version>1.0</version> 
%
\setlength{\unitlength}{3947sp}%
\begin{picture}(7383,1665)(976,-1319)
%  METADATA <id>134</id> 
\put(8196,-66){\makebox(0,0)[lb]{\smash{\fontsize{6}{7.2}\usefont{T1}{ptm}{m}{it}{\color[rgb]{0,0,0}$r=0$}%
}}}
%  METADATA <id>28</id> 
\put(4758,-1239){\makebox(0,0)[lb]{\smash{\fontsize{10}{12}\usefont{T1}{ptm}{m}{it}{\color[rgb]{0,0,0}$(S)$}%
}}}
%  METADATA <id>88</id> 
\put(4321,-66){\makebox(0,0)[lb]{\smash{\fontsize{6}{7.2}\usefont{T1}{ptm}{m}{it}{\color[rgb]{0,0,0}$w>0$}%
}}}
%  METADATA <id>99</id> 
\put(5411,-76){\makebox(0,0)[lb]{\smash{\fontsize{6}{7.2}\usefont{T1}{ptm}{m}{it}{\color[rgb]{0,0,0}$r>0$}%
}}}
%  METADATA <id>21</id> 
\put(3207,-1248){\makebox(0,0)[lb]{\smash{\fontsize{10}{12}\usefont{T1}{ptm}{m}{it}{\color[rgb]{0,0,0}$(S_0)$}%
}}}
%  METADATA <id>92</id> 
\put(2976,-91){\makebox(0,0)[lb]{\smash{\fontsize{6}{7.2}\usefont{T1}{ptm}{m}{it}{\color[rgb]{0,0,0}$w>0, r=0$}%
}}}
%  METADATA <id>50</id> 
\put(1713,-1254){\makebox(0,0)[lb]{\smash{\fontsize{10}{12}\usefont{T1}{ptm}{m}{it}{\color[rgb]{0,0,0}$(R)$}%
}}}
%  METADATA <id>86</id> 
\put(1326,-76){\makebox(0,0)[lb]{\smash{\fontsize{6}{7.2}\usefont{T1}{ptm}{m}{it}{\color[rgb]{0,0,0}$w=0$}%
}}}
%  METADATA <id>38</id> 
\put(991,-334){\makebox(0,0)[lb]{\smash{\fontsize{6}{7.2}\usefont{T1}{ptm}{m}{it}{\color[rgb]{0,0,0}$b=1$}%
}}}
%  METADATA <id>97</id> 
\put(991,-479){\makebox(0,0)[lb]{\smash{\fontsize{6}{7.2}\usefont{T1}{ptm}{m}{it}{\color[rgb]{0,0,0}$r=0$}%
}}}
%  METADATA <id>104</id> 
\put(6303,-1234){\makebox(0,0)[lb]{\smash{\fontsize{10}{12}\usefont{T1}{ptm}{m}{it}{\color[rgb]{0,0,0}$(S)$}%
}}}
%  METADATA <id>112</id> 
\put(6971,-71){\makebox(0,0)[lb]{\smash{\fontsize{6}{7.2}\usefont{T1}{ptm}{m}{it}{\color[rgb]{0,0,0}$r>0$}%
}}}
%  METADATA <id>127</id> 
\put(7733,-1229){\makebox(0,0)[lb]{\smash{\fontsize{10}{12}\usefont{T1}{ptm}{m}{it}{\color[rgb]{0,0,0}$(S)$}%
}}}
\end{picture}%

%% file: growth_classCr.pdftex_t
\begin{picture}(0,0)%
\includegraphics{growth_classCr.pdf}%
\end{picture}%
%
%  Created by WinFIG version 2024.2 
%  METADATA <version>1.0</version> 
%
\setlength{\unitlength}{3947sp}%
\begin{picture}(10701,6957)(232,-6349)
%  METADATA <id>82</id> 
\put(7834,-652){\makebox(0,0)[lb]{\smash{\fontsize{12}{14.4}\usefont{T1}{ptm}{m}{it}{\color[rgb]{0,0,0}$1$}%
}}}
%  METADATA <id>74</id> 
\put(7974,-1766){\makebox(0,0)[lb]{\smash{\fontsize{12}{14.4}\usefont{T1}{ptm}{m}{it}{\color[rgb]{0,0,0}$(1,1,r+b-1)'_L$}%
}}}
%  METADATA <id>79</id> 
\put(8033,-3949){\makebox(0,0)[lb]{\smash{\fontsize{12}{14.4}\usefont{T1}{ptm}{m}{it}{\color[rgb]{0,0,0}$(1,0,0)_R$}%
}}}
%  METADATA <id>118</id> 
\put(9846,-6274){\makebox(0,0)[lb]{\smash{\fontsize{12}{14.4}\usefont{T1}{ptm}{m}{it}{\color[rgb]{0,0,0}$(r)'$}%
}}}
%  METADATA <id>69</id> 
\put(1429,-1774){\makebox(0,0)[lb]{\smash{\fontsize{12}{14.4}\usefont{T1}{ptm}{m}{it}{\color[rgb]{0,0,0}$(b,w,r)'_C$}%
}}}
%  METADATA <id>71</id> 
\put(3484,-132){\makebox(0,0)[lb]{\smash{\fontsize{12}{14.4}\usefont{T1}{ptm}{m}{it}{\color[rgb]{0,0,0}$w+1$}%
}}}
%  METADATA <id>77</id> 
\put(3744,-3939){\makebox(0,0)[lb]{\smash{\fontsize{12}{14.4}\usefont{T1}{ptm}{m}{it}{\color[rgb]{0,0,0}$(1,w,0)_S$}%
}}}
%  METADATA <id>81</id> 
\put(5262,-402){\makebox(0,0)[lb]{\smash{\fontsize{12}{14.4}\usefont{T1}{ptm}{m}{it}{\color[rgb]{0,0,0}$1$}%
}}}
%  METADATA <id>83</id> 
\put(3747,-2262){\makebox(0,0)[lb]{\smash{\fontsize{12}{14.4}\usefont{T1}{ptm}{m}{it}{\color[rgb]{0,0,0}$w$}%
}}}
%  METADATA <id>85</id> 
\put(3717,-4370){\makebox(0,0)[lb]{\smash{\fontsize{12}{14.4}\usefont{T1}{ptm}{m}{it}{\color[rgb]{0,0,0}$w$}%
}}}
%  METADATA <id>86</id> 
\put(3019,-5194){\makebox(0,0)[lb]{\smash{\fontsize{12}{14.4}\usefont{T1}{ptm}{m}{it}{\color[rgb]{0,0,0}$b+1$}%
}}}
%  METADATA <id>87</id> 
\put(3661,-6249){\makebox(0,0)[lb]{\smash{\fontsize{12}{14.4}\usefont{T1}{ptm}{m}{it}{\color[rgb]{0,0,0}$(b+1,w,r)'_C$}%
}}}
%  METADATA <id>6</id> 
\put(934,-784){\makebox(0,0)[lb]{\smash{\fontsize{12}{14.4}\usefont{T1}{ptm}{m}{it}{\color[rgb]{0,0,0}$b$}%
}}}
%  METADATA <id>7</id> 
\put(1351,-136){\makebox(0,0)[lb]{\smash{\fontsize{12}{14.4}\usefont{T1}{ptm}{m}{it}{\color[rgb]{0,0,0}$w$}%
}}}
%  METADATA <id>70</id> 
\put(3631,-1764){\makebox(0,0)[lb]{\smash{\fontsize{12}{14.4}\usefont{T1}{ptm}{m}{it}{\color[rgb]{0,0,0}$(1,w+1,r)'_L$}%
}}}
%  METADATA <id>73</id> 
\put(5551,-1764){\makebox(0,0)[lb]{\smash{\fontsize{12}{14.4}\usefont{T1}{ptm}{m}{it}{\color[rgb]{0,0,0}$(1,1,r+1)'_L$}%
}}}
%  METADATA <id>78</id> 
\put(5731,-3946){\makebox(0,0)[lb]{\smash{\fontsize{12}{14.4}\usefont{T1}{ptm}{m}{it}{\color[rgb]{0,0,0}$(1,0,0)_R$}%
}}}
%  METADATA <id>106</id> 
\put(5803,-6274){\makebox(0,0)[lb]{\smash{\fontsize{12}{14.4}\usefont{T1}{ptm}{m}{it}{\color[rgb]{0,0,0}$(1)'$}%
}}}
%  METADATA <id>112</id> 
\put(7686,-6274){\makebox(0,0)[lb]{\smash{\fontsize{12}{14.4}\usefont{T1}{ptm}{m}{it}{\color[rgb]{0,0,0}$(1)$}%
}}}
%  METADATA <id>121</id> 
\put(7029,-849){\makebox(0,0)[lb]{\smash{\fontsize{20}{24}\usefont{T1}{ptm}{m}{n}{\color[rgb]{0,0,0}$\ldots$}%
}}}
%  METADATA <id>123</id> 
\put(7096,-2814){\makebox(0,0)[lb]{\smash{\fontsize{20}{24}\usefont{T1}{ptm}{m}{n}{\color[rgb]{0,0,0}$\ldots$}%
}}}
%  METADATA <id>125</id> 
\put(8814,-5236){\makebox(0,0)[lb]{\smash{\fontsize{20}{24}\usefont{T1}{ptm}{m}{n}{\color[rgb]{0,0,0}$\ldots$}%
}}}
%  METADATA <id>128</id> 
\put(2611,-189){\makebox(0,0)[lb]{\smash{\fontsize{12}{14.4}\usefont{T1}{ptm}{m}{it}{\color[rgb]{0,0,0}$r$}%
}}}
\end{picture}%

%% file: growth_classLr.pdftex_t
\begin{picture}(0,0)%
\includegraphics{growth_classLr.pdf}%
\end{picture}%
%
%  Created by WinFIG version 2024.2 
%  METADATA <version>1.0</version> 
%
\setlength{\unitlength}{3947sp}%
\begin{picture}(8251,3412)(929,-2806)
%  METADATA <id>159</id> 
\put(7096,-1911){\makebox(0,0)[lb]{\smash{\fontsize{12}{14.4}\usefont{T1}{ptm}{m}{it}{\color[rgb]{0,0,0}$\ldots$}%
}}}
%  METADATA <id>144</id> 
\put(6794,293){\makebox(0,0)[lb]{\smash{\fontsize{8}{9.6}\usefont{T1}{ptm}{m}{it}{\color[rgb]{0,0,0}$r$}%
}}}
%  METADATA <id>152</id> 
\put(4919, 63){\makebox(0,0)[lb]{\smash{\fontsize{8}{9.6}\usefont{T1}{ptm}{m}{it}{\color[rgb]{0,0,0}$r$}%
}}}
%  METADATA <id>155</id> 
\put(9165,-1579){\makebox(0,0)[lb]{\smash{\fontsize{8}{9.6}\usefont{T1}{ptm}{m}{it}{\color[rgb]{0,0,0}$r$}%
}}}
%  METADATA <id>69</id> 
\put(1191,-1056){\makebox(0,0)[lb]{\smash{\fontsize{12}{14.4}\usefont{T1}{ptm}{m}{it}{\color[rgb]{0,0,0}$(1,w,r)'_L$}%
}}}
%  METADATA <id>7</id> 
\put(976, 94){\makebox(0,0)[lb]{\smash{\fontsize{8}{9.6}\usefont{T1}{ptm}{m}{it}{\color[rgb]{0,0,0}$w>0$}%
}}}
%  METADATA <id>117</id> 
\put(4106,-2721){\makebox(0,0)[lb]{\smash{\fontsize{12}{14.4}\usefont{T1}{ptm}{m}{it}{\color[rgb]{0,0,0}$(1)'$}%
}}}
%  METADATA <id>124</id> 
\put(5996,-2731){\makebox(0,0)[lb]{\smash{\fontsize{12}{14.4}\usefont{T1}{ptm}{m}{it}{\color[rgb]{0,0,0}$(1)$}%
}}}
%  METADATA <id>138</id> 
\put(2319, 38){\makebox(0,0)[lb]{\smash{\fontsize{8}{9.6}\usefont{T1}{ptm}{m}{it}{\color[rgb]{0,0,0}$r$}%
}}}
%  METADATA <id>131</id> 
\put(8167,-2708){\makebox(0,0)[lb]{\smash{\fontsize{12}{14.4}\usefont{T1}{ptm}{m}{it}{\color[rgb]{0,0,0}$(r)'$}%
}}}
%  METADATA <id>92</id> 
\put(3556,-1071){\makebox(0,0)[lb]{\smash{\fontsize{12}{14.4}\usefont{T1}{ptm}{m}{it}{\color[rgb]{0,0,0}$(1,w+1,r)'_L$}%
}}}
%  METADATA <id>97</id> 
\put(3516, 44){\makebox(0,0)[lb]{\smash{\fontsize{8}{9.6}\usefont{T1}{ptm}{m}{it}{\color[rgb]{0,0,0}$w+1$}%
}}}
%  METADATA <id>99</id> 
\put(5840,-1066){\makebox(0,0)[lb]{\smash{\fontsize{12}{14.4}\usefont{T1}{ptm}{m}{it}{\color[rgb]{0,0,0}$(2,w,r)'_C$}%
}}}
%  METADATA <id>104</id> 
\put(5589,338){\makebox(0,0)[lb]{\smash{\fontsize{8}{9.6}\usefont{T1}{ptm}{m}{it}{\color[rgb]{0,0,0}$w$}%
}}}
\end{picture}%

%% file: growth_classRr.pdftex_t
\begin{picture}(0,0)%
\includegraphics{growth_classRr.pdf}%
\end{picture}%
%
%  Created by WinFIG version 2024.2 
%  METADATA <version>1.0</version> 
%
\setlength{\unitlength}{3947sp}%
\begin{picture}(7431,1877)(393,-1149)
%  METADATA <id>122</id> 
\put(5184,-1066){\makebox(0,0)[lb]{\smash{\fontsize{12}{14.4}\usefont{T1}{ptm}{m}{it}{\color[rgb]{0,0,0}$(1,1,0)_R$}%
}}}
%  METADATA <id>7</id> 
\put(408,-72){\makebox(0,0)[lb]{\smash{\fontsize{10}{12}\usefont{T1}{ptm}{m}{it}{\color[rgb]{0,0,0}$w=0$}%
}}}
%  METADATA <id>116</id> 
\put(3451,-1074){\makebox(0,0)[lb]{\smash{\fontsize{12}{14.4}\usefont{T1}{ptm}{m}{it}{\color[rgb]{0,0,0}$(1,1,0)_L$}%
}}}
%  METADATA <id>128</id> 
\put(6886,-1066){\makebox(0,0)[lb]{\smash{\fontsize{12}{14.4}\usefont{T1}{ptm}{m}{it}{\color[rgb]{0,0,0}$(2,0,0)_{C_1}$}%
}}}
%  METADATA <id>69</id> 
\put(1149,-1066){\makebox(0,0)[lb]{\smash{\fontsize{12}{14.4}\usefont{T1}{ptm}{m}{it}{\color[rgb]{0,0,0}$(1,0,0)_R$}%
}}}
%  METADATA <id>137</id> 
\put(2781,-45){\makebox(0,0)[lb]{\smash{\fontsize{10}{12}\usefont{T1}{ptm}{m}{it}{\color[rgb]{0,0,0}$w=1$}%
}}}
\end{picture}%

%% file: growth_classSr.pdftex_t
\begin{picture}(0,0)%
\includegraphics{growth_classSr.pdf}%
\end{picture}%
%
%  Created by WinFIG version 2024.2 
%  METADATA <version>1.0</version> 
%
\setlength{\unitlength}{3947sp}%
\begin{picture}(7657,5979)(937,-5368)
%  METADATA <id>242</id> 
\put(6676,-4314){\makebox(0,0)[lb]{\smash{\fontsize{12}{14.4}\usefont{T1}{ptm}{m}{it}{\color[rgb]{0,0,0}$\ldots$}%
}}}
%  METADATA <id>229</id> 
\put(7757,-5293){\makebox(0,0)[lb]{\smash{\fontsize{12}{14.4}\usefont{T1}{ptm}{m}{it}{\color[rgb]{0,0,0}$(r)'$}%
}}}
%  METADATA <id>164</id> 
\put(7099,-1273){\makebox(0,0)[lb]{\smash{\fontsize{12}{14.4}\usefont{T1}{ptm}{m}{it}{\color[rgb]{0,0,0}$(2,w,r)'_C$}%
}}}
%  METADATA <id>145</id> 
\put(3215, 73){\makebox(0,0)[lb]{\smash{\fontsize{10}{12}\usefont{T1}{ptm}{m}{it}{\color[rgb]{0,0,0}$w+1$}%
}}}
%  METADATA <id>146</id> 
\put(3325,-1267){\makebox(0,0)[lb]{\smash{\fontsize{12}{14.4}\usefont{T1}{ptm}{m}{it}{\color[rgb]{0,0,0}$(1,w+1,r)'_L$}%
}}}
%  METADATA <id>177</id> 
\put(3409,-3403){\makebox(0,0)[lb]{\smash{\fontsize{12}{14.4}\usefont{T1}{ptm}{m}{it}{\color[rgb]{0,0,0}$(1,w,r+1)'_S$}%
}}}
%  METADATA <id>186</id> 
\put(5131,-3403){\makebox(0,0)[lb]{\smash{\fontsize{12}{14.4}\usefont{T1}{ptm}{m}{it}{\color[rgb]{0,0,0}$(1,1,r+1)'_S$}%
}}}
%  METADATA <id>176</id> 
\put(3467,-2059){\makebox(0,0)[lb]{\smash{\fontsize{10}{12}\usefont{T1}{ptm}{m}{it}{\color[rgb]{0,0,0}$w$}%
}}}
%  METADATA <id>185</id> 
\put(4994,-2075){\makebox(0,0)[lb]{\smash{\fontsize{10}{12}\usefont{T1}{ptm}{m}{it}{\color[rgb]{0,0,0}$1$}%
}}}
%  METADATA <id>154</id> 
\put(5129, 90){\makebox(0,0)[lb]{\smash{\fontsize{10}{12}\usefont{T1}{ptm}{m}{it}{\color[rgb]{0,0,0}$w$}%
}}}
%  METADATA <id>155</id> 
\put(5197,-1273){\makebox(0,0)[lb]{\smash{\fontsize{12}{14.4}\usefont{T1}{ptm}{m}{it}{\color[rgb]{0,0,0}$(1,w,0)_S$}%
}}}
%  METADATA <id>163</id> 
\put(7031,320){\makebox(0,0)[lb]{\smash{\fontsize{10}{12}\usefont{T1}{ptm}{m}{it}{\color[rgb]{0,0,0}$w$}%
}}}
%  METADATA <id>7</id> 
\put(1115, 84){\makebox(0,0)[lb]{\smash{\fontsize{10}{12}\usefont{T1}{ptm}{m}{it}{\color[rgb]{0,0,0}$w$}%
}}}
%  METADATA <id>69</id> 
\put(1177,-1267){\makebox(0,0)[lb]{\smash{\fontsize{12}{14.4}\usefont{T1}{ptm}{m}{it}{\color[rgb]{0,0,0}$(1,w,r)'_S$}%
}}}
%  METADATA <id>240</id> 
\put(2285, 48){\makebox(0,0)[lb]{\smash{\fontsize{10}{12}\usefont{T1}{ptm}{m}{it}{\color[rgb]{0,0,0}$r$}%
}}}
%  METADATA <id>213</id> 
\put(3865,-5287){\makebox(0,0)[lb]{\smash{\fontsize{12}{14.4}\usefont{T1}{ptm}{m}{it}{\color[rgb]{0,0,0}$(1)'$}%
}}}
%  METADATA <id>221</id> 
\put(5515,-5293){\makebox(0,0)[lb]{\smash{\fontsize{12}{14.4}\usefont{T1}{ptm}{m}{it}{\color[rgb]{0,0,0}$(1)$}%
}}}
\end{picture}%